\documentclass{article}
\usepackage[english]{babel}
\usepackage[utf8]{inputenc}
\usepackage[T1]{fontenc}
\usepackage[toc,page]{appendix}
\usepackage{graphicx}
\usepackage{geometry}
\usepackage{sectsty}
\usepackage{bbm}
\usepackage{dsfont}
\usepackage{alltt}
\usepackage{verbatimbox}
\usepackage{amsthm,amsfonts}
\usepackage{amsmath}
\usepackage{mathabx}
\usepackage{accents}
\usepackage{titlesec}
\usepackage{mathtools}
\mathtoolsset{showonlyrefs=true}

\titleformat*{\section}{\Large\bfseries}
\titleformat*{\subsection}{\large\bfseries}
\titleformat*{\subsubsection}{\large\bfseries}
\titleformat*{\paragraph}{\large\bfseries}
\titleformat*{\subparagraph}{\large\bfseries}

\newtheorem{teo}{Theorem}[section]
\newtheorem{lema}[teo]{Lemma}
\newtheorem{cor}[teo]{Corollary}
\newtheorem{prop}[teo]{Proposition}
\newtheorem{defi}[teo]{Definition}
\newtheoremstyle{mytheoremstyle} % name
{\topsep}                    % Space above
{\topsep}                    % Space below
{}                   % Body font
{}                           % Indent amount
{\scshape}                   % Theorem head font
{.}                          % Punctuation after theorem head
{.5em}                       % Space after theorem head
{}  % Theorem head spec (can be left empty, meaning ‘normal’)

\theoremstyle{mytheoremstyle} \newtheorem{nota}{Remark}[section]
\theoremstyle{mytheoremstyle} \newtheorem{exemplo}{Example}[section]
\theoremstyle{mytheoremstyle} \newtheorem*{notacao}{Notation}
\numberwithin{equation}{section}
\newcommand{\real}{\mathbb{R}}
\newcommand{\complex}{\mathbb{C}}
\newcommand{\nat}{\mathbb{N}}

\newcommand \ben {\begin{equation}}
\newcommand \een {\end{equation}}
\newcommand \be {\begin{equation*}}
\newcommand \ee {\end{equation*}}
\newcommand \bi {\begin{itemize}}
\newcommand \ei {\end{itemize}}
\newcommand{\ubar}[1]{\underaccent{\bar}{#1}}

\DeclareMathOperator*{\parteim}{Im}

\title{\textbf{Spatial plane waves for the nonlinear Schrödinger equation: local existence and stability results}}
\author{Simão Correia and Mário Figueira}
\begin{document}

\maketitle
\begin{abstract}
We consider the Cauchy problem for the nonlinear Schrödinger equation on $\real^2$, $iu_t + u_{xx} + u_{yy} + \lambda|u|^\sigma u =0$, $\lambda\in \real$, $\sigma>0$. We introduce new functional spaces over which the initial value problem is well-posed. Their construction is based on \textit{spatial plane waves} (cf. \cite{correiafigueira}). These spaces contain $H^1(\real^2)$ and do not lie within $L^2(\real^2)$. We prove several global well-posedness and stability results over these new spaces, including a new global well-posedness result of $H^1$ solutions with indefinitely large $H^1$ and $L^2$ norms. Some of these results are proved using a new functional transform, the \textit{plane wave transform}. We develop a suitable theory for this transform, prove several properties and solve classical linear PDE's with it, highlighting its wide range of application.
\end{abstract}
\section{Introduction}
In this work, we consider the initial value problem for the nonlinear Schrödinger equation in two spatial dimensions
\begin{equation}\label{NLS}\tag{NLS}
\left\{\begin{array}{l}
iu_t + u_{xx} + u_{yy} + \lambda|u|^\sigma u =0, \quad u=u(t,x,y),\ (x,y)\in \real^2\\
u(0,x,y)=u_0(x,y),\quad \sigma>0,\ \lambda\in \real.
\end{array}\right..
\end{equation}
This classical equation has been studied intensively for the past fifty years and many problems regarding local well-posedness, global existence, blowup and asymptotic behaviour have very complete answers (see the monographs \cite{cazenave}, \cite{sulem}, \cite{tao} and references therein). The most standard framework where one studies this initial value problem is based on the following $H^1$ local well-posedness result: for any $u_0\in H^1(\real^2)$, there exists a unique maximal solution $u\in C([0,T(u_0)), H^1(\real^2))$ of \eqref{NLS}, which depends continuously on the initial data. Moreover, if $T(u_0)<\infty$, one has $\|\nabla u(t)\|_2\to \infty$ as $t\to T(u_0)$. The $H^1$ framework allows one to used certain conserved quantities (mass, energy, variance, etc.) to obtain precise results on the dynamical behaviour of solutions. Apart from this, few nonstandard frameworks have been developed (see, for example, \cite{gallo}, for a theory on Zhidkov spaces).

Here, we shall introduce new functional spaces and study local well-posedness, global existence and stability (in a way that will be clear further ahead). These spaces do not lie within $L^2(\real^2)$: in fact, the first one will lie in $L^\infty(\real^2)$, but not in $L^p(\real^2)$, for any $p<\infty$; the second will lie in $L^4(\real^2)\cap L^\infty(\real^2)$, but not in $L^p(\real^2)$, $p<4$. The most simple example of these spaces appeared for the first time in a previous work, \cite{correiafigueira}, in the context of the hyperbolic nonlinear Schrödinger equation. We shall generalize the results of \cite{correiafigueira} in the context of the (NLS) for two reasons: to our knowledge, there are no results of this type for (NLS); secondly, even though one may prove similar results for more general equations, that would increase needlessly the complexity of the exposition and deviate the attention of the reader from the essential ideas.

The first step is to take any $c\in\real$ and consider solutions of (NLS) of the form $u(t,x,y)=f(t,x-cy)$ (we call these solutions \textit{spatial plane waves}: the number $c$ is the \textit{speed} and the function $f$ is the \textit{profile} of the wave). Introducing this ansatz, one arrives to
$$
if_t + (1+c^2)f_{zz} + \lambda|f|^\sigma f=0,
$$
which one may solve using the standard $H^1$ theory. It is clear that spatial plane waves are not $L^p$ solutions of (NLS), for any $p<\infty$, even though they lie in $L^1_{loc}(\real^2)$. Based on these solutions, consider the space of spatial plane waves
\begin{equation}
X_c=\left\{u\in L^1_{loc}(\real^2): \exists f\in H^1(\real): u(x,y)=f(x-cy)\ a.e. \right\}.
\end{equation}
Then the initial value problem for (NLS) is well-posed over $X_c$, simply because the problem reduces to the local theory in one spatial dimension. Now consider
$$
E=H^1(\real^2)\oplus X_c.
$$
Notice that, due to the lack of decay of elements in $X_c$, the sum is indeed a direct sum. Then the initial value problem for $u_0=v_0+\phi_0\in E$ can be proven to be equivalent to the initial value problem for the system
\begin{equation}
\left\{\begin{array}{l}
iv_t + v_{xx} + v_{yy} + \lambda|v+\phi|^\sigma(v+\phi) - \lambda|\phi|^\sigma\phi=0,\ \phi(t,x,y)=f(t,x-cy)\\
v(0)=v_0\\
if_t + (1+c^2)f_{zz} + \lambda|f|^\sigma f=0\\
f(0)=f_0
\end{array}\right.
\end{equation}
where $f_0$ is the profile of $\phi_0$. Notice that the second equation is independent on $v$. This means that one may solve it and introduce the solution onto the first equation. Thus, to prove local well-posedness on $E$, it suffices to show an $H^1$ local well-posedness result for the first equation. This is achieved by making the crucial observation that the nonlinear terms lie in $H^{-1}$ (even though they involve $\phi$).

With such a local well-posedness result, consider the following situation: take a spatial plane wave $\phi$ with initial data $\phi_0$, and assume $\phi$ is globally defined. If one introduces an $H^1$ perturbation on the initial data, $\phi_0+\epsilon v_0$, the corresponding solution is given by $u=\phi+v_\epsilon$ (since $\phi$ only depends on the plane wave part of the initial data). Then a natural question is the following: 
\begin{center}
\textit{Stability problem: if $\epsilon$ is sufficiently small, does $u$ stay close to $\phi$?}
\end{center}

 This question corresponds to a global existence result for $v$ that insures that $v$ stays small for all times. Results of such nature are indeed valid for (NLS) in $H^1$ (see \cite{cazenaveweissler}, \cite{hayashi}, \cite{hayashisaitoh}, \cite{kato1}, \cite{nakamuraozawa}) in the $L^2$-supercritical case. Here, however, the question is not that clear: first, if one develops the nonlinear terms in the equation for $v$, one sees that lower order terms in $v$ are present; second, the plane wave component $\phi$, which even has "infinite mass" and "infinite energy" (due to the lack of integrability) might act as a forcing term, pushing $v$ to grow indefinitely. In \cite{correiafigueira}, such a result was proven for $\sigma$ an even power greater or equal to 4, in the context of the hyperbolic nonlinear Schrödinger equation.
\vskip15pt
Now we advance to the next level of complexity. Take a sequence of wave speeds $\ubar{c}=\{c_n\}_{n\in \nat}$, with $c_i\neq c_j$, $i\neq j$, and consider the space
\begin{align*}
X_{\ubar{c}}=\Big\{\phi\in L^1_{loc}(\real^2): \phi(x,y)=\sum_{n\ge 1} f_n(x-c_ny), ((1+c_n^2)f_n)_{n\in\nat}\in l^1(H^2(\real)) \Big\}.
\end{align*}
endowed with the induced norm (which, as we will check, is well-defined). As before, we look for a local well-posedness result in
$$
E_{\ubar{c}}=H^1(\real^2)\oplus X_{\ubar{c}},
$$
which turns out to be equivalent to a local well-posedness result for the infinite system
\begin{equation}
\left\{\begin{array}{l}
iv_t + v_{xx} + v_{yy} + \lambda|v+\phi|^\sigma(v+\phi) - \sum_{n\ge 1}\lambda|\phi_n|^\sigma\phi_n=0,\quad \phi_n(t,x,y)=f_n(t,x-c_ny)\\

i((f_n)_t + (1+c_n^2)(f_n)_{zz} + \lambda|f_n|^\sigma f_n=0, \ n\in\nat\\

\end{array}\right..
\end{equation}
The equations for the profiles $f_n$ are solved using the $H^2(\real)$ local well-posedness result (here, one must be careful to ensure that \textit{all} profiles exist up to some time $T>0$). Once again, it remains to prove that one may solve the equation for $v$ in $H^1(\real^2)$, which amounts to check that the nonlinear terms are in $H^{-1}(\real^2)$. Under the additional hypothesis $\sigma\ge 1$, this can be shown to be true: heuristically, if one develops the nonlinear terms, one obtains some terms with $v$, which are well-behaved, and some products of diferent $\phi_n$'s. Remarkably, these terms lie in $L^2(\real^2)$: take, for example, $\phi_j\phi_k$, $j\neq k$. Then
\begin{align*}
\left(\int_{\real^2} |\phi_j(x,y)|^{2}|\phi_k(x,y)|^2dxdy\right)^{1/2}&=\left(\int_{\real^2} |f_j(x-c_jy)|^{2}|f_k(x-c_ky)|^2dxdy\right)^{1/2}\\&= \frac{1}{|c_j-c_k|^{1/2}}\left(\int_{\real^2} |f_j(w)|^{2}|f_k(z)|^2dwdz\right)^{1/2} \\& \le  \frac{1}{|c_j-c_k|^{1/2}}\|f_j\|_{L^2}\|f_k\|_{L^2}.
\end{align*}
However, since one has an infinite sum of such terms, one must be able to control
$$
\sum_{j\neq k} \frac{1}{|c_j-c_k|^{1/2}}\|f_j\|_{L^2}\|f_k\|_{L^2}.
$$
The above quantity turns out to be preserved by the flow of the infinite system, since one has conservation of the $L^2$ norm of each $f_n$, and therefore it can controlled by the initial data. Hence we shall restrict ourselves to
$$
A_{\ubar{c}}=H^1(\real^2)\oplus Y_{\ubar{c}},\quad Y_{\ubar{c}}=\left\{\phi\in X_{\ubar{c}}: \sum_{i\neq j}\frac{\|f_j\|_{L^2}\|f_k\|_{L^2}}{|c_j-c_k|^{1/2}}<\infty \right\}.
$$
One finally concludes that it is possible to solve the equation for $v$ and arrives to the local well-posedness result over $A_{\ubar{c}}$:
\begin{teo}\label{existnumeravel}
Fix $\sigma\ge 1$ and $u_0\in A_{\ubar{c}}$. Then there exists a unique maximal solution $u\in C([0,T), A_{\ubar{c}})\cap C^1((0,T),E_{\ubar{c}}')$ (cf. \eqref{defiEc}) of (NLS) such that $u(0)=u_0$. The solution depends continuously of $u_0$. Furthermore, if $T<\infty$, then
$$
\|u(t)\|_{E_{\ubar{c}}}\to\infty,\quad t\to T.
$$
\end{teo}
As in the single plane wave case, the stability problem may be studied. Here, we present the following result:
\begin{teo}\label{estabilidadenumeravel}
Set $\sigma\ge4$ even. Given $M>1$, there exist $\epsilon(M)>0$ and $\delta=\delta(\epsilon,M)$, with $\delta(\epsilon,M)\to0$ as $\epsilon\to 0$, such that, if $\phi_0\in Y_{\ubar{c}}$ and $v_0\in H^1(\real^2)$ satisfy
$$
\sum_{n\ge 1}\|z(f_0)_n\|_{L^2} + \|\partial_z (f_n)_0\|_{L^1} + \sum_{j\neq k} \frac{(1+c_j^2)^{1/2}\|\partial_z(f_0)_j\|_{L^2} + \|(f_0)_j\|_{L^2}}{|c_j-c_k|^{1/2}}\|(f_0)_k\|_{L^2}<M,
$$
$$
\|\phi_0\|_{X_{\ubar{c}}}<\epsilon, \|v_0\|_{H^1}<\epsilon,\ \epsilon<\epsilon(M),
$$
then the solutions of (NLS) $u$ and $\tilde{u},$ with initial data $v_0+\phi_0$ and $\phi_0$, respectively, are both globally defined and satisfy
$$
\|u-\tilde{u}\|_{L^\infty((0,\infty); H^1(\real^2))}\le \delta(\epsilon,M).
$$
\end{teo}

\vskip15pt

We now arrive to the last generalization, more complex and with interesting properties. The idea is simply to pass from a discrete sum of plane waves to an \textit{integration over a continuum of plane waves}. That is, we consider
$$
X=\left\{\phi\in L^1_{loc}(\real^2): \phi(x,y)=\int_{\real} f(x-cy,c), f\in L_c^1(\real; H_z^1(\real))\cap L^\infty_c(\real,L^2_z(\real)) \right\},
$$
endowed with the norm $\|\phi\|=\|f\|_{L^1_c(H^1_z)}+\|f\|_{L^\infty_c(L^2_z)}$. One may actually define an integral transform, the \textit{plane wave transform}, mapping functions in two \textit{speed} variables to functions in two \textit{physical} variables
\begin{equation}\label{transformadaT}
(Tf)(x,y)=\int f(x-cy,c)dc.
\end{equation}
Before we study the initial value problem for (NLS) in this context, we shall study some very interesting properties related to the plane wave transform. Among them, one may prove that:
\begin{itemize}
\item $Tf\in L^p(\real^2)$, $p\ge 2$ under suitable conditions on $f$ (Proposition \ref{integrabilidadeLp} and Corollary \ref{integrabilidadeL2});
\item $Tf\notin L^2(\real^2)$ if $f\in C_0(\real^2)$ is positive (Corollary \ref{naointegrabilidadeL2});
\item The convolution of two functions, the Fourier transform and the Laplace transform may be obtained using the plane wave transform transform (Corollary \ref{convolucao} and Examples \ref{calor} and \ref{schrodinger});
\item Several classical linear equations, such as the heat equation, the Schrödinger equation and the wave equation, may be solved by means of this transform (Section \ref{resolverequacoes}).
\end{itemize}

\begin{nota}
One should regard all the previous definitions in analogy to the Fourier series and transform: one starts with a simple periodic function with a given frequency, superposes a numerable family of functions with different frequencies to arrive to the Fourier series and passes to the continuous case to obtain the Fourier transform, where one also has the concepts of physical and frequency variables. 

Evidently, our construction is not the same as the Fourier one and many properties will differ. One aspect is that, while the Fourier construction makes all sense in one variable (and its multidimensional version is simply the application of the one-dimensional case to each variable), the plane wave construction needs (at least) a two-dimensional setting. In another perspective, the Fourier transform is based on the solutions of the ODE
$$
u''(x) + k^2u(x) = 0,
$$
while the plane wave theory is based on solutions of the transport equation
$$
u_y(x,y)+cu_x(x,y)=0.
$$
\end{nota}

With this new transform in hand, we try to obtain some results in the spirit of the numerable case. That is, look for local well-posedness and stability results in $E=H^1(\real^2) + X$. The main difference is that now the (NLS) may be decoupled into
\begin{equation}
\left\{\begin{array}{l}
iv_t + v_{xx} + v_{yy} + \lambda|v+\phi|^\sigma(v+\phi) =0\\
v(0)=v_0\\
i\phi_t + \phi_{xx} + \phi_{yy}=0\\
\phi(0)=\phi_0
\end{array}\right..
\end{equation}
Surprisingly, when one passes to the continuous case, there exists a decoupling such that the equation for the plane wave component is linear. We now state our main results.

\begin{teo}\label{existenciacontinuo}
Fix $\sigma\ge 1$ and $u_0\in E$. Then there exists a unique maximal solution $u\in C([0,T), E)$ of (NLS) (cf. definition \ref{defisolucao}) such that $u(0)=u_0$. The solution depends continuously on the initial data. Furthermore, if $T<\infty$, then
$$
\|u(t)\|_E\to\infty,\quad t\to T.
$$
\end{teo}

\begin{teo}\label{estabilidadecontinuo}
Set $\sigma\ge 4$ even. Given $M>1$, there exist $\epsilon(M)>0$ and $\delta=\delta(\epsilon, M)$, with $\delta(\epsilon, M)\to 0$ as $\epsilon\to 0$, such that, if $\phi_0=Tf_0\in X$ and $v_0\in H^1(\real^2)$ satisfy
$$
\left\|\frac{f_0}{(1+c^2)^{1/2}}\right\|_{L^1_c(L^1_z)} + \left\|(f_0)_z\right\|_{L^1_c(L^1_z)}+ \|f_0\|_{L^1_c(H^2_z)} + \|cf_0\|_{L^1_c(H^2_z)}<M,
$$
$$
\|\phi_0\|_{X}<\epsilon, \|v_0\|_{H^1}<\epsilon,\ \epsilon<\epsilon(M),
$$
then the solution $u$ of (NLS) with initial data $v_0+\phi_0$ is globally defined. Moreover, if $S_2$ is the free Schrödinger group in dimension two,
$$
\|u-S_2\phi_0\|_{L^\infty((0,\infty),H^1(\real^2))}\le \delta(\epsilon,M)
$$ 
\end{teo}
As a consequence of Theorem \ref{estabilidadecontinuo}, we have
\begin{teo}\label{grandesdados}
Set $\sigma\ge 4$ even. For any given $K>0$ and $M>1$, there exists $\phi_0\in X\cap H^1(\real^2)$ satisfying the conditions of Theorem \ref{estabilidadecontinuo} and
$$
\|\phi_0\|_{L^2(\real^2)}>K,\quad \|\nabla \phi_0\|_{L^2(\real^2)}>K.
$$
Consequently, for any $v_0\in H^1(\real^2)$ such that
$$
\|v_0\|_{H^1}<\epsilon(M)
$$ 
the $H^1$ solution $u$ of (NLS) with initial data $v_0+\phi_0$ is global and satisfies
$$
\|u-S_2\phi_0\|_{L^\infty((0,\infty), H^1(\real^2)}\le \delta(\epsilon(M), M).
$$
\end{teo}

\begin{nota}
The global $H^1$ solutions given by Theorem \ref{grandesdados} have small $L^{\sigma+2}$ norm, since this norm is controlled by $\|v_0\|_{H^1} + \|\phi_0\|_X$. Therefore the energy of these solutions will always be positive and so there is no possible contradiction with the usual Virial blowup result by Glassey \cite{glassey}.
\end{nota}

Finally, we present a global well-posedness result for $\sigma=1$.

\begin{teo}\label{gwp}
Fix $\sigma=1$. Consider $u_0=v_0+\phi_0$, $v_0\in H^1(\real^2)$, $\phi_0=Tf_0\in X$. Assume that
$$
(1+c^3)(1+|z|)f_0\in L^1_c(H^2_z)\cap L^\infty_c(H^2_z).
$$
Then there exists a unique global solution $u\in C([0,\infty), E)$ of (NLS) with initial data $u_0$.
\end{teo}

We now outline the structure of the paper. Section 2 will be devoted to the study of the numerable superposition of plane waves and Theorems \ref{existnumeravel} and \ref{estabilidadenumeravel} are proved therein. In section 3, we recall the plane wave transform \eqref{transformadaT}, construct a suitable functional framework in which one may consider such a transform, and prove several properties. In section 4, we apply the transform to solve some classic linear PDE's. In section 5, we apply the theory for the transform to obtain Theorems \ref{existenciacontinuo}, \ref{estabilidadecontinuo}, \ref{grandesdados} and \ref{gwp}. We conclude with some remarks and, in the appendix, we give an alternate proof of the injectivity of the plane wave transform.

\begin{notacao}
We fix some notations. We define the space of Schwarz funtions, $\mathcal{S}(\real^2)$, and the space of distributions, $\mathcal{D}'(\real^2)$. We denote by $\mathcal{F}_{*}$ the Fourier transform with respect to the $\ast$ variables. Integration spaces will often be indexed with the variable that is being integrated: $L^1_c(\real, L^2_z(\real))$ will denote the space of integrable functions in $c\in\real$ with values in $L^2(\real)$, where the variable is $z$. Furthermore, if no confusion can arise, we omit the domain in the notation of the integration space.
\end{notacao}

%\begin{nota}
%In the above theorem, suppose furthermore that $\phi_0\in H^1(\real^2)$. This implies that the solution of (NLS) with initial data $v_0+\phi_0$ is actually an $H^1$ solution. 
%\end{nota}
%
%Now we try to obtain the analogous results for $E=H^1(\real^2) + X$.
%
% There are, however, some differences from the discrete case:
%\begin{enumerate}
%\item To ensure that the norm in $X$ is well-defined, one has to prove the uniqueness of representation of elements in $X$. That is,
%$$
%\int_{\real} f(c,x-cy)=0, \forall\ (x,y)\in \real^2\ a.e. \Rightarrow f=0\ a.e.
%$$
%The proof of such a result is not trivial at all and one needs to take a smaller profile space:
%\begin{prop}
%UNICIDADE DE REPRESENTAÇÃO
%\end{prop}
%\item For any $p>2$, if $f$ satisfies
%$$
%\int \int \frac{1}{|c-c'|^{2/p}}\|f(c)\|_{L_z^{p/2}}\|f(c')\|_{L_z^{p/2}}dcdc'<\infty,
%$$
%then $\phi\in L^p(\real^2)$. This is in sharp contrast with the discrete case, where no sum of spatial plane waves is in $L^p(\real^2)$, $p<\infty$. Notice that, when $p= 2$, the only possible profile verifying the above condition is the zero function.
%\item It is not known that $H^{1}(\real^2)\cap X$ is trivial. If such an intersection is nontrivial, it implies a lack of uniqueness of solution of (NLS) over $E$. In an attempt to answer this question, we present the following result:
%\begin{prop}
%POSITIVAS E A PARTICULAR.
%\end{prop}
%\end{enumerate}

\section{Plane waves: the numerable case}

Let us recall the numerable construction. Fix a sequence of wave speeds $\ubar{c}=\{c_n\}_{n\in \nat}$, with $c_i\neq c_j$, $i\neq j$. Define the spaces
\begin{align}
X_{\ubar{c}}=\Big\{\phi\in L^1_{loc}(\real^2): \phi(x,y)=\sum_{n\ge 1} f_n(x-c_ny),\nonumber ((1+c_n^2)f_n)_{n\in\nat}\in l^1(H^2(\real)),  \Big\}\label{defiXc}
\end{align}
and
\begin{equation}
X'_{\ubar{c}}=\left\{\phi\in L^1_{loc}(\real^2): \phi(x,y)=\sum_{n\ge 1} f_n(x-c_ny), (f_n)_{n\in\nat}\in l^1(L^2(\real)) \right\}\label{defiXclinha}.
\end{equation}
\begin{lema}\label{injectiva}
If $\ubar{f}\in l^1(L^2(\real))$ is such that
$$
\phi(x,y)=\sum_{n\ge 1} f_n(x-c_ny)=0,\quad x,y\in\real^2,
$$
then $\ubar{f}\equiv 0$.
\end{lema}
\begin{proof}
Fix $k\in\nat$, $h\in\real$ and define
$$
\phi_h(x)=\phi(x+c_kh,h) = f_k(x) + \sum_{n\ge 1,n\neq k} f_n(x + (c_k-c_n)h).
$$
By absurd, suppose that there exists an open interval $]a,b[$ such that
$$
\|f_k\|_{L^2(]a,b[)} =\delta >0.
$$
Since $\ubar{f}\in l^1(L^2(\real))$, there exists $n_0\in\nat$ such that
$$
\sum_{n\ge n_0} \|f_n\|_{L^2} < \delta/2.
$$
On the other hand, for $h$ large enough, one has
$$
\sum_{n= 1,n\neq k}^{n_0} \|f_n(\cdot + (c_k-c_n)h)\|_{L^2(]a,b[)} < \delta/2.
$$
Hence
$$
\delta=\|f_k\|_{L^2(]a,b[)}  \le \sum_{n\ge 1,n\neq k} \|f_n(\cdot + (c_k-c_n)h)\|_{L^2(]a,b[)} <\delta,
$$
which is impossible.
\end{proof}

Since each element of $X_{\ubar{c}}$ and $X_{\ubar{c}}'$ may be represented in a unique way, we define the norms of these spaces as the norm induced by the profile space:
$$
\|\phi\|_{X_{\ubar{c}}}:=\sum_{n\ge 1} \|(1+c_n^2)f_n\|_{H^2},\quad \|\phi\|_{X_{\ubar{c}}'}:=\sum_{n\ge 1} \|f_n\|_{L^2}.
$$

\begin{lema}\label{somadirecta}
Consider $H^{-1}(\real^2)$ and $X'_{\ubar{c}}$ as subspaces of $\mathcal{D}'(\real^2)$. Then $H^{-1}(\real^2)\cap X'_{\ubar{c}}=\{0\}$.
\end{lema}
\begin{proof}
Suppose that $\ubar{f}\in l^1(L^2(\real))$ is such that the function $\phi$ defined by
$$
\phi(x,y)=\sum_{n\ge 1} f_n(x-c_ny)
$$
is in $H^{-1}(\real^2)$. Fix $k\in\nat$ and define
$$
\phi_h(x,y)=\phi(x+c_kh,y+h) = f_k(x-c_ky) + \sum_{n\ge 1,n\neq k} f_n(x -c_ny+ (c_k-c_n)h).
$$
Given $\psi\in C^\infty_0(\real^2)$, it is easy to check that, since $\phi\in H^{-1}(\real^2)$,
$$
\langle \phi_h, \psi\rangle_{H^{-1}\times H^1} \to 0,\quad h\to\infty.
$$
Through a similar argument to that of the previous proof, one may check that, when $h\to \infty$,
$$
\int \left(f_k(x-c_ky) + \sum_{n\ge 1,n\neq k} f_n(x -c_ny+ (c_k-c_n)h)\right)\psi(x,y)dxdy \to \int f_k(x-c_ky)\psi(x,y)dxdy.
$$
Hence one has necessarily
$$
\int f_k(x-c_ky)\psi(x,y)dxdy=0,\forall \psi\in C_0^\infty(\real^2)
$$
and so $f_k\equiv 0$. Since $k\in\nat$ is arbitrary, one concludes $\phi\equiv 0$.
\end{proof}
Define the following subspaces of $\mathcal{D}'(\real^2)$:
\begin{equation}\label{defiEc}
E_{\ubar{c}}=H^1(\real^2) \oplus X_{\ubar{c}},\quad E_{\ubar{c}}'=H^{-1}(\real^2)\oplus X_{\ubar{c}}'.
\end{equation}
Notice that it follows from the previous lemma that these sums are indeed direct sums. Moreover, consider
\begin{equation}
Y_{\ubar{c}}=\left\{\phi\in X_{\ubar{c}}: \sum_{i\neq j}\frac{\|f_j\|_{L^2}\|f_k\|_{L^2}}{|c_j-c_k|^{1/2}}<\infty \right\},\quad A_{\ubar{c}}=H^1(\real^2)\oplus Y_{\ubar{c}}.
\end{equation}

\begin{lema}\label{decomposicao}
Fix $\sigma\ge 1$. Take $v\in H^1(\real^2)$ and $\phi\in Y_{\ubar{c}}$. Writing
$$
\phi(x,y)=\sum_{n\ge 1} f_n(x-c_ny),
$$
define $\phi_n(x,y)=f_n(x-c_ny)$ and
$$
g=\sum_{n\ge 1}|\phi_n|^\sigma \phi_n
$$
Then $g\in X_{\ubar{c}}'$ and $|v+\phi|^{\sigma}(v+\phi)-g\in H^{-1}(\real^2)$.
\end{lema}
\begin{proof}
Since, for any $k\in\nat$,
$$
\|\phi_k\|_{L^\infty} \le \sum_{n\ge 1}\|\phi_n\|_{L^\infty} \le \sum_{n\ge 1}\|f_n\|_{L^\infty} \le \|\phi\|_{X_{\ubar{c}}},
$$
one has
$$
\|g\|_{X_{\ubar{c}}'}\le \left(\sup_k \|\phi_k\|_\infty^\sigma\right)\left\|\sum_{n\ge 1} |\phi_n|\right\|_{X_{\ubar{c}}'} \le \|\phi\|_{X_{\ubar{c}}}^\sigma \left\|\sum_{n\ge 1} |\phi_n|\right\|_{X_{\ubar{c}}'} \lesssim \|\phi\|_{X_{\ubar{c}}}^{\sigma+1}
$$
and so $g\in X_{\ubar{c}}'$. For the second part of the result, recall the classical estimates
$$
||a|^\sigma a - |b|^\sigma b|\le (\sigma +1)(|a|^\sigma + |b|^\sigma)|a-b|,\ a,b\in\complex
$$
and
$$
||a+b|^{\sigma}(a+b) - |a|^\sigma - |b|^\sigma|\lesssim |a|^\sigma |b| + |b|^\sigma |a|, \ a,b\in\complex.
$$
This implies that
\begin{align*}
&\left|\left|v+\sum_{k\ge 1} \phi_k\right|^{\sigma}\left(v+\sum_{k\ge 1} \phi_k\right)-g\right|\le \left|\left|v+\sum_{k\ge 1} \phi_k\right|^\sigma \left(v+\sum_{k\ge 1} \phi_k\right) - \left|\sum_{k\ge 1} \phi_k\right|^\sigma \left(\sum_{k\ge 1} \phi_k\right)\right| \\ +& \sum_{j\ge 1}\left|\left|\sum_{k\ge j}\phi_k\right|^\sigma\left(\sum_{k\ge j}\phi_k\right) - \left|\sum_{k\ge j+1} \phi_k\right|^\sigma\left(\sum_{k\ge j+1} \phi_k\right) - |\phi_j|^\sigma \phi_j\right|\\\lesssim &\left(|v|^\sigma + \left|\sum_{k\ge 1} \phi_k\right|^\sigma\right)|v| + \sum_{j\ge 1} \left(\left|\sum_{k\ge j+1}\phi_k\right|^\sigma|\phi_j| + \left|\sum_{k\ge j+1}\phi_k\right||\phi_j|^\sigma \right)\\
\lesssim & |v|^{\sigma+1} +\sum_{k\ge 1} \|\phi_k\|_{L^\infty}^\sigma|v| + \sum_{j\ge 1} \sum_{k\ge j+1}\left(\|\phi_k\|_{L^\infty}^{\sigma-1} + \|\phi_j\|_{L^\infty}^{\sigma-1}\right)|\phi_j\phi_k|
\end{align*}
The first term, $|v|^{\sigma+1}$, is in $L^{\frac{\sigma+2}{\sigma+1}}$, by Sobolev's injection. The second term is clearly in $L^2$. Finally, we prove that the third term is also in $L^2$:
\begin{align*}
\left\|\sum_{j\ge 1} \sum_{k\ge j+1}\left(\|\phi_k\|_{L^\infty}^{\sigma-1} + \|\phi_j\|_{L^\infty}^{\sigma-1}\right)|\phi_j\phi_k|\right\|_{L^2} &\lesssim \|\phi\|_{X_{\ubar{c}}}^{\sigma-1} \sum_{j\ge 1} \sum_{k\ge j+1}\|\phi_j\phi_k\|_{L^2} \\&= \|\phi\|_{X_{\ubar{c}}}^{\sigma-1} \sum_{j\neq k}\frac{\|f_j\|_{L^2}\|f_k\|_{L^2}}{|c_j-c_k|^{1/2}}<\infty.
\end{align*}
Hence $|v+\phi|^\sigma(v+\phi)-g\in L^{\frac{\sigma+2}{\sigma+1}} + L^2 \hookrightarrow H^{-1}(\real^2)$.
\end{proof}

\begin{proof}[Proof of Theorem \ref{existnumeravel}]
\textit{Step 1.} Write $u_0=v_0+\phi_0$, $v_0\in H^1(\real^2)$, $\phi_0\in Y_{\ubar{c}}$,
$$
\phi_0(x,y)=\sum_{n\ge 1} (f_0)_n(x-c_ny).
$$
We claim that, over $A_{\ubar{c}}$, the initial value problem for (NLS) is equivalent to the initial value problem
\begin{align}
& iv_t + v_{xx} + v_{yy} + \lambda|v+\phi|^\sigma(v+\phi) - \sum_{n\ge 1}\lambda|\phi_n|^\sigma\phi_n=0,\quad \phi_n(t,x,y)=f_n(t,x-c_ny)\label{eqv}\\
& v(0)=v_0\nonumber\\
& i((f_n)_t + (1+c_n^2)(f_n)_{zz} + \lambda|f_n|^\sigma f_n=0, \ n\in\nat\label{sistemainfinito}\\
& f_n(0)=(f_0)_n\nonumber
\end{align}
Indeed, if $v$ and $(f_n)_{n\in\nat}$ are solutions of this problem, it is trivial to check that
$$
u(t,x,y)=v(t,x,y)+\sum_{n\ge 1} f_n(t,x-c_ny)
$$
is a solution of (NLS). On the other hand, suppose that $u$ is a solution of (NLS) with initial condition $u_0$. Decompose $u$ as $v+\phi$ and write
$$
\phi(t,x,y) = \sum_{n\ge 1} f_n(t,x-c_ny),\quad g(t,x,y)=\sum_{n\ge 1} \left(|f_n|^\sigma f_n\right)(t,x-c_ny)
$$
Then
$$
iv_t+v_{xx}+v_{yy} + \lambda |v+\phi|^\sigma(v+\phi) - \lambda g = -\left(i\phi_t + \phi_{xx} + \phi_{yy} + \lambda g\right)
$$
Since the left hand side is in $H^{-1}(\real^2)$ and the right hand side is in $X_{\ubar{c}}'$, by lemma \ref{somadirecta}, both sides must be equal to zero:
$$
\left\{\begin{array}{l}
iv_t + v_{xx} + v_{yy} + \lambda|v+\phi|^\sigma(v+\phi) - \lambda g=0\\
v(0)=v_0\\
i\phi_t + \phi_{xx} + \phi_{yy} + \lambda g=0\\
\phi(0)=\phi_0
\end{array}\right..
$$
Furthermore, by lemma \ref{injectiva}, the second equation is equivalent to the infinite system $\eqref{sistemainfinito}$, which proves the claim.

\textit{Step 2.} We solve the infinite system \eqref{sistemainfinito}. For each $n\in \nat$, define
$$
(h_0)_n(z)=(f_0)_n(\sqrt{1+c_n^2}z).
$$
Then
$$
\|h_n(0)\|_{H^2}\le (1+c_n^2)^{3/4}\|(f_0)_n\|_{H^2}.
$$
Since $((1+c_n^2)(f_0)_n)\in l^1(H^2)$,  $(1+c_n^2)^{3/4}\|(f_0)_n\|_{H^2}\to 0$ as $n\to\infty$.

Now consider the initial value problem
$$
i(h_n)_t + (h_n)_{zz} + \lambda|h_n|^\sigma h_n=0, h_n(0)=(h_n)_0.
$$
It follows from the $H^2$ local well-posedness results for (NLS) that there exists a time $T_n$ and a unique maximal solution $h_n\in C([0,T_n), H^2(\real))\cap C^1([0,T_n), L^2(\real))$ of the above problem. Moreover, since $\|h_n(0)\|_{H^2}\to 0$, for $n\ge n_0$ sufficiently large, $T_n\ge T_1$. Define the common time of existence,
$$
T_\infty =\min_{n\in\nat} T_n. 
$$
Then each $h_n$ is define on $[0,T_\infty)$. If $T_\infty<\infty$, then, as $t\to T_\infty$, 
$$
\|h_n(t)\|_{H^2}\to \infty, \mbox{ for some } n\in\nat.
$$
Setting 
$$f_n(t,\cdot)=h_n\left(t,\frac{\cdot}{\sqrt{1+c_n^2}}\right),\ 0\le t< T_{\infty},
$$
it is clear that the sequence $(f_n)_{n\in\nat}$ is the unique solution of \eqref{sistemainfinito} on $[0,T_\infty)$. Moreover, considering that
$$
\sup_n \|h_n(t)\|_{H^2}\le \sup_n (1+c_n^2)^{3/4}\|f_n(t)\|_{H^2} \le \left\|\left((1+c_n^2) f_n(t)\right)_{n\in\nat}\right\|_{l^1(H^2)},
$$
if $T_\infty<\infty$, one has
$$
\left\|\left((1+c_n^2) f_n(t)\right)_{n\in\nat}\right\|_{l^1(H^2)}\to \infty,\quad t\to T_{\infty}.
$$
Finally, it follows from the conservation of the $L^2$ norm that
$$
\left(\sum_{j\neq q} \frac{\|f_j(t)\|_{L^2}\|f_k(t)\|_{L^2}}{|c_j-c_k|^{1/2}}\right)^{1/2} = \left(\sum_{j\neq q} \frac{\|(f_0)_j\|_{L^2}\|(f_0)_k\|_{L^2}}{|c_j-c_k|^{1/2}}\right)^{1/2},\ 0\le t< T_\infty.
$$
Setting
$$
\phi(t,x,y)=\sum_{n\ge 1} f_n(t,x-c_ny),
$$
this implies that $\phi\in C([0,T_\infty),Y_{\ubar{c}})\cap C^1((0,T_\infty), X_{\ubar{c}}')$.

\textit{Step 3.} We now solve \eqref{eqv}. By Lemma 
\ref{decomposicao} and Kato's local well-posedness result in $H^1(\real^2)$, \eqref{eqv} has a unique maximal solution $v$, defined on $[0,T_v)$, $T_v\le T_\infty$, with initial condition $v_0$. Furthermore, if $T_v<T_\infty$,
$$
\|v(t)\|_{H^1}\to\infty ,\quad t\to T_v.
$$
\textit{Step 4.} Conclusion. It follows from the previous steps that $u=v+\phi$ is the unique solution of (NLS) on $E$, with initial data $u_0$. If $T_v<\infty$, then either $T_v<T_\infty$ or $T_\infty<\infty$. In any case,
$$
\|u(t)\|_E = \|v(t)\|_{H^1} + \|\phi(t)\|_{X_{\ubar{c}}}\to\infty ,\quad t\to T_{max},
$$
which implies that $u$ is not extendible over $E$. The continuous dependence on the initial data is a consequence of the continuous dependence given by the local well-posedness results for $v$ and $(f_n)_{n\in\nat}$.
\end{proof}

\begin{lema}\label{decaimento}
Set $\sigma=4$. Given $M>1$, there exists $\epsilon=\epsilon(M)>0$ such that, given $\phi_0\in Y_{\ubar{c}}$ satisfying
\begin{equation}
N(\phi_0)=\sum_{n\ge 1}\|z(f_0)_n\|_{L^2} + \|\partial_z (f_n)_0\|_{L^1}<M,\ \|\phi_0\|_{X_{\ubar{c}}}<\epsilon
\end{equation} 
the solution $\phi$ of (NLS) with initial data $\phi_0$ is global and satisfies
$$
\|\phi(t)\|_{L^\infty}\lesssim \min\left\{\epsilon, \frac{M}{t^{1/2}} \right\},\ \|\nabla \phi(t)\|_{L^\infty}\lesssim M^3,\ \|\phi(t)\|_{X_{\ubar{c}}}\le \frac{3}{2}\|\phi_0\|_{X_{\ubar{c}}}, \ t>0.
$$

\end{lema}
\begin{proof}
We write
$$
\phi(t,x,y)=\sum_{n\ge1} \phi_n(t,x,y),\quad \phi_n(t,x,y) = f_n(t,x-c_ny),\quad  \phi_0(x,y)=\sum_{n\ge 1} (f_0)_n(x-c_ny)
$$
where $f_n$ is a solution of
$$
i(f_n)_t + (1+c_n^2)(f_n)_{zz} + \lambda |f_n|^4f_n =0,\quad f_n(0)=(f_0)_n.
$$
Using the rescaling $h_n(t,z)=f_n(t,\sqrt{1+c_n^2}z)$, one arrives to
$$
i(h_n)_t + (h_n)_{zz} + \lambda |h_n|^4f_n =0,\quad  h_n(0,z)=(h_0)_n(z)=(f_0)_n(\sqrt{1+c_n^2}z).
$$
\textit{Step 1.} We now collect some properties of $h_n$. 

First of all, it follows from \cite[Theorem 5.3.1]{cazenave} that there exists $\epsilon_0>0$ such that, if $\|(h_0)_n\|_{H^1}<\epsilon_0$, then $h_n$ is global and
$$
\|h_n(t)\|_{H^1}\le \frac{3}{2}\|(h_0)_n\|_{H^1},\ t>0.
$$
From decay estimates, which are valid regardless of the sign of $\lambda$, one also has 
$$
\|h_n(t)\|_{L^\infty} \le \frac{C}{t^{1/2}}\|z(h_0)_n\|_{L^2}^{1/2}\|(h_0)_n\|_{L^2}^{1/2},\ t>0.
$$
These properties imply that
$$
\|h_n(t)\|_{L^\infty} \lesssim \min\left\{\|(h_0)_n\|_{H^1}, \frac{1}{t^{1/2}}\|z(h_0)_n\|_{L^2}^{1/2}\|(h_0)_n\|_{L^2}^{1/2}\right\}.
$$
Now we obtain an uniform estimate for $\|\partial_z h_n(t)\|_{L^\infty}$. Recall, from the proof of Theorem \ref{existnumeravel}, that there exists $T>0$ (independent of $n$) such that $h_n$ is defined on $[0,T]$ and $\|h_n(t)\|_{H^2}\le 2\|(h_n)_0\|_{H^2}$, for $0<t<T$. Hence
$$
\|\partial_z h_n(t)\|_{L^\infty} \le 2\|(h_n)_0\|_{H^2}, \ 0<t<T.
$$
For $t>T$, since
$$
\partial_z h_n (t)= e^{it\partial^2_{zz}}\partial_z (h_0)_n + \int_0  e^{i(t-s)\partial^2_{zz}}\partial_z \left(|h_n(s)|^4|h_n(s)|\right),
$$
one has
\begin{align*}
\|\partial_z h_n(t)\|_{L^\infty}\lesssim & \|\partial_z (h_n)_0\|_{L^1} + \int_0^t \frac{1}{\sqrt{t-s}}\|h_n(s)\|_{L^\infty}^3\|h_n(s)\|_{L^2} \|\partial_z h_n(s)\|_{L^2} ds\\
\lesssim & \|\partial_z (h_n)_0\|_{L^1} + \int_0^T \frac{1}{\sqrt{t-s}}\|h_n(s)\|_{H^1}^5 ds \\ &\quad + \int_T^t \frac{1}{\sqrt{t-s}}\frac{1}{s^{3/2}}\|z(h_0)_n\|_{L^2}^{3/2}\|(h_0)_n\|_{L^2}^{3/2}\|h_n(s)\|_{H^1}^2 ds\\
\lesssim &\|\partial_z (h_n)_0\|_{L^1} + C((h_0)_n)\bigg(\int_0^1 \frac{1}{\sqrt{t-s}} ds +\int_1^{t-1} \frac{1}{\sqrt{t-s}}\frac{1}{s^{3/2}} ds \\ &\quad + \int_{t-1}^t \frac{1}{\sqrt{t-s}}\frac{1}{s^{3/2}} ds \bigg)\\
\lesssim & \|\partial_z (h_n)_0\|_{L^1}+ C((h_0)_n)\left(\frac{1}{\sqrt{t-1}} + \int_1^{t-1}\frac{1}{s^{3/2}}ds + \frac{1}{(t-1)^{3/2}}\int_{t-1}^{t}\frac{1}{\sqrt{t-s}}\right)\\
\lesssim & \|\partial_z (h_n)_0\|_{L^1} + C((h_0)_n)\left(\frac{1}{\sqrt{t-1}} + 1 + \frac{1}{(t-1)^{3/2}}\right) \le \|\partial_z (h_n)_0\|_{L^1} + C((h_0)_n).
\end{align*}
where  $$C((h_0)_n)= \|(h_0)_n\|_{H^1}^5 + \|z(h_0)_n\|_{L^2}^{3/2}\|(h_0)_n\|_{L^2}^{3/2}\|(h_0)_n\|_{H^1}^2.$$

Thus we obtain the estimate
$$
\|\partial_z h_n(t)\|_{L^\infty} \lesssim \|(h_0)_n\|_{H^2} + \|\partial_z (h_n)_0\|_{L^1} + C((h_0)_n)
$$

\textit{Step 2.} Now we write the estimates of Step 1 in terms of $\phi$. Some simple computations show that, for any $1\le q\le \infty$,
$$
\|h_n(t)\|_{L^q}=(1+c_n^2)^{-1/2q}\|f_n(t)\|_{L^q}, \quad \|\partial_z h_n(t)\|_{L^q}=(1+c_n^2)^{-1/2q+1/2}\|\partial_z f_n(t)\|_{L^q}
$$
$$
\|\partial_{zz} h_n(t)\|_{L^q}=(1+c_n^2)^{-1/2q+1}\|\partial_{zz} f_n(t)\|_{L^q},\quad \|zh_n(t)\|_{L^2} = (1+c_n^2)^{-3/4}\|zf_n(t)\|_{L^2}
$$
Notice that, if $\|\phi_0\|_{X_{\ubar{c}}}<\epsilon<M$, then
$$
\|(h_n)_0\|_{H^2}\le \sum_{n\ge1} \|(h_n)_0\|_{H^2} \le \sum_{n\ge 1} (1+c_n^2)^{3/4}\|(f_n)_0\|_{H^2}\le \epsilon.
$$

Hence
\begin{align*}
\|\phi(t)\|_{L^\infty} &\le \sum_{n\ge 1} \|f_n(t)\|_{L^\infty} = \sum_{n\ge 1} \|h_n(t)\|_{L^\infty} \lesssim \sum_{n\ge 1} \min\left\{\|(h_0)_n\|_{H^1}, \frac{1}{t^{1/2}}\|z(h_0)_n\|_{L^2}^{1/2}\|(h_0)_n\|_{L^2}^{1/2}\right\} \\&\le \min\left\{\sum_{n\ge 1} \|(f_0)_n\|_{H^1}, \frac{1}{t^{1/2}}\sum_{n\ge 1}\|z(f_0)_n\|_{L^2}^{1/2}\|(f_0)_n\|_{L^2}^{1/2}\right\}\\&\le \min\left\{\epsilon, \frac{1}{t^{1/2}} \left(\sum_{n\ge 1}\|z(f_0)_n\|_{L^2}+\|(f_0)_n\|_{L^2}\right)\right\}\le \min\left\{\epsilon, \frac{M+\epsilon}{t^{1/2}} \right\}\lesssim \min\left\{\epsilon, \frac{M}{t^{1/2}}\right\}
\end{align*}
and, in a similar fashion,
\begin{align*}
\|\nabla \phi(t)\|_{L^\infty} &\le \sum_{n\ge 1} (1+c_n^2)^{1/2}\|\partial_z f_n(t)\|_{L^\infty} = \sum_{n\ge 1} \|\partial_z h_n(t)\|_{L^\infty} \\&\lesssim \sum_{n\ge 1} \|(h_0)_n\|_{H^2} + \|\partial_z (h_n)_0\|_{L^1} + \|(h_0)_n\|_{H^1}^5 + \|z(h_0)_n\|_{L^2}^{3/2}\|(h_0)_n\|_{L^2}^{3/2}\|(h_0)_n\|_{H^1}^2\\&\lesssim \epsilon + M + \sum_{n\ge 1} \|z(h_0)_n\|_{L^2}^{3}+\|(h_0)_n\|_{L^2}^{3}\|(h_0)_n\|_{H^1}^4\lesssim \epsilon+ M+M^3\lesssim M^3.
\end{align*}

\end{proof}

\begin{proof}[Proof of Theorem \ref{estabilidadenumeravel}]
We shall only prove the case $\sigma=4$, the general being completely analogous. The main idea of the proof is to obtain a "small data global existence" result for the $H^1$ component (see, for example, \cite[Theorem 6.2.1]{cazenave}).

\textit{Step 1. Setup.}
From the previous lemma, it follows that, for $\epsilon>0$ sufficiently small, $\phi$ is global and
$$
\|\phi(t)\|_{L^\infty}\lesssim \min\left\{\epsilon, \frac{M}{t^{1/2}} \right\}\lesssim M,\ \|\nabla \phi(t)\|_{L^\infty}\lesssim M^3, \|\phi(t)\|_{X_{\ubar{c}}}\lesssim \epsilon,\ t>0.
$$

Fix $v_0\in H^1(\real^2)$ and consider the corresponding solution $v$ of 
\begin{equation}\label{equacaov}
iv_t +  v_{xx} + v_{yy} + \lambda(|v+\phi|^4(v+\phi) -\sum_{n\ge 1} |\phi_n|^4\phi_n)=0,\quad \phi_n(t,x,y)=f_n(t,x-c_ny).
\end{equation}
We recall that $v$ is defined on $(0,T(u_0))$, where $u_0=v_0+\phi_0$. Since $\phi$ is global in $X_{\ubar{c}}$, the blow-up alternative of Theorem \ref{existnumeravel} then implies that, if $T(u_0)<\infty$,
$$
\|v(t)\|_{H^1}\to \infty, \ t\to T(u_0).
$$ 

We develop the nonlinear part as
$$
\lambda(|v+\phi|^4(v+\phi) -\sum_{n\ge 1} |\phi_n|^4\phi_n)=\sum_{i=0}^5 g_i(v,\phi),
$$
where each $g_i$ has exactly $i$ powers of $v$. Define, for $i=3,4,5$, $\rho_i=i+1$ and $\gamma_i$ such that $(\gamma_i,\rho_i)$ is an admissible pair, that is, $2/\gamma_i = 1- 2/\rho_i$. In particular, $\rho_3=\gamma_3=4$. Consider, for $0<t<T(u_0)$,
\begin{equation}\label{funcaoh}
h(t)=\|v\|_{L^\infty((0,t),H^1(\real^2))} + \sum_{i=3}^{5} \|v\|_{L^{\gamma_i}((0,t),W^{1,\rho_i}(\real^2))}.
\end{equation}
We write Duhamel's formula,
$$
v(t)=U(t)v_0 + \sum_{i=0}^{5}\int_0^t U(t-s)g_i(v(s),\phi(s))ds.
$$
Therefore, for any admissible pair $(q,r)$,
\begin{equation}\label{estimativa}
\|v\|_{L^q((0,t),W^{1,r}(\real^2))}\le C\|v_0\|_{H^1} + \sum_{i=0}^{5} \left\|\int_0^\cdot U(\cdot-s)g_i(v(s),\phi(s))ds\right\|_{L^q((0,t),W^{1,r}(\real^2))}.
\end{equation}
For the sake of simplicity, we shall omit both the temporal and spatial domains. In the next steps, we shall estimate each term of the sum by a suitable power of $h(t)$. All constants depending solely on $M$ shall be ommited.

\textit{Step 2. Estimate of higher-order terms in $v$ on \eqref{estimativa}.} Here, we shall estimate
$$
\left\|\int_0^\cdot U(\cdot-s)g_i(v(s),\phi(s))ds\right\|_{L^q(W^{1,r})},\quad i=3,4,5.
$$
Take $i=3$. Then it follows from Step 1 that
\begin{align*}
\left\|\int_0^\cdot U(\cdot-s)g_3(v(s),\phi(s))ds\right\|_{L^q(W^{1,r})}&\lesssim \|g_3(v,\phi)\|_{L^{\gamma_3'}(W^{1,\rho_3'})} \\ &\lesssim \||v|^3|\phi|^2\|_{L^{4/3}(W^{1,4/3})}\lesssim \|v\|^3_{L^4(W^{1,4})} \\&\lesssim \|v\|_{L^{\gamma_3}(W^{1,\rho_3})}^3 \lesssim h(t)^3. 
\end{align*}

Now we treat the case $i=4,5$:
\begin{align*}
\left\|\int_0^\cdot U(\cdot-s)g_i(v(s),\phi(s))ds\right\|_{L^q(W^{1,r})}&\lesssim \|g_i(v,\phi)\|_{L^{\gamma_i'}(W^{1,\rho_i'})} \\&\lesssim \||v|^i|\phi|^{5-i}\|_{L^{\gamma_i'}(W^{1,\rho_i'})} \lesssim \|v\|_{L^{\mu_i}(L^{\rho_i})}^{i-1}\|v\|_{L^{\gamma_i}(W^{1,\rho_i})}
\end{align*}
where
$$
\mu_i=\frac{(i-1)(i+1)}{2}>\gamma_i.
$$
Then, through the interpolation $L^{\gamma_i} - L^{\mu_i} - L^\infty$ and the injection $H^1\hookrightarrow L^{\rho_i}$,
$$
\left\|\int_0^\cdot U(\cdot-s)g_i(v(s),\phi(s))ds\right\|_{L^q(W^{1,r})}\lesssim h(t)^i, \ i=4,5.
$$

\textit{Step 3. Estimate of the linear term in $v$.}
\begin{align*}
\left\|\int_0^\cdot U(\cdot-s)g_1(v(s),\phi(s))ds\right\|_{L^q(W^{1,r})}\lesssim \|g_1(v,\phi)\|_{L^1(H^1)}\lesssim \||v||\phi|^4\|_{L^1(H^1)}.
\end{align*}
Using the properties deduced in Step 1,
\begin{align*}
\||v||\phi|^4\|_{L^1(H^1)}&\lesssim \int_0^t \|\phi(s)\|_{L^\infty}^3\|\phi(s)\|_{W^{1,\infty}}\|v(s)\|_{H^1} ds\\&\lesssim \|v\|_{L^\infty(H^1)}\|\phi\|_{L^\infty(W^{1,\infty})}\|\phi\|_{L^\infty(L^\infty)}^{1/2}\left(\int_0^t \|\phi(s)\|_{L^\infty}^{3/2} ds\right) \\&\lesssim
\|v\|_{L^\infty(H^1)}\|\phi\|_{L^\infty(W^{1,\infty})}\|\phi\|_{L^\infty(L^\infty)}^{1/2}\left(1+ \int_1^t \frac{1}{s^{3/2}} ds\right) \\&\lesssim \epsilon^{1/2}\|v\|_{L^\infty(H^1)}\lesssim \epsilon^{1/2}h(t).
\end{align*}
\textit{Step 4. Estimate of the term independent on $v$.}
Define $$D=\{(j,k,l,m,n)\in \nat^5: (k,l,m,n)\neq (j,j,j,j)\}$$. Then
\begin{align*}
&\left\|\int_0^\cdot U(\cdot-s)g_0(v(s),\phi(s))ds\right\|_{L^q(W^{1,r})}\lesssim \|g_0(v,\phi)\|_{L^1(H^1)}= \left\||\phi|^4\phi - \sum_{n\ge 1} |\phi_n|^4\phi_n\right\|_{L^1(H^1)} \\=\ & \left\|\sum_{ (j,k,l,m,n)\notin D} \phi_j\overline{\phi_k}\phi_l\overline{\phi_m}\phi_n\right\|_{L^1(H^1)}\le \sum_{ (j,k,l,m,n)\notin D} \int_0^t\|(\phi_j\overline{\phi_k}\phi_l\overline{\phi_m}\phi_n)(s)\|_{H^1}ds\\\le\ &\int_0^t \left(\sum_{l,m,n\ge 1} \|\phi_l(s)\|_{L^\infty}\|\phi_m(s)\|_{L^\infty}\|\phi_n(s)\|_{L^\infty}\right)\left(\sum_{j\neq k} \|\nabla\phi_j(s)\phi_k(s)\|_{L^2} + \|\phi_j(s)\phi_k(s)\|_{L^2}\right)ds\\\le \ & \int_0^t \left(\sum_{n\ge 1}\|\phi_n(s)\|_{L^\infty}\right)^3\left(\sum_{j\neq k} \frac{(1+c_j^2)^{1/2}}{|c_j-c_k|^{1/2}}\|\partial_zf_j(s)\|_{L^2}\|f_k(s)\|_{L^2} + \frac{\|f_j(s)\|_{L^2}\|f_k(s)\|_{L^2}}{|c_j-c_k|^{1/2}} \right)ds.
\end{align*}
Recalling that 
$$\|f_k(s)\|_{2}=\|(f_0)_k\|_{2}, \|\partial_z f_k(s)\|_{2}\le \frac{3}{2}\|\partial_z (f_0)_k\|_{2},\ k\in\nat, s>0,
$$
one estimates
\begin{align*}
&\sum_{j\neq k} \frac{(1+c_j^2)^{1/2}}{|c_j-c_k|^{1/2}}\|\partial_zf_j(s)\|_{L^2}\|f_k(s)\|_{L^2} + \frac{\|f_j(s)\|_{L^2}\|f_k(s)\|_{L^2}}{|c_j-c_k|^{1/2}} \\\le\ & \frac{3}{2}\sum_{j\neq k} \frac{(1+c_j^2)^{1/2}}{|c_j-c_k|^{1/2}}\|\partial_z(f_0)_j\|_{L^2}\|(f_0)_k\|_{L^2} + \frac{\|(f_0)_j\|_{L^2}\|(f_0)_k\|_{L^2}}{|c_j-c_k|^{1/2}}\lesssim M.
\end{align*}
Hence, by Lemma \ref{decaimento},
\begin{align*}
&\left\|\int_0^\cdot U(\cdot-s)g_0(v(s),\phi(s))ds\right\|_{L^q(W^{1,r})}\lesssim \int_0^t \left(\sum_{n\ge 1}\|\phi_n(s)\|_{L^\infty}\right)^3 \\\lesssim\ & 
\|\phi\|_{L^\infty(L^\infty)}^{1/2}\int_0^t \left(\sum_{n\ge 1}\|\phi_n(s)\|_{L^\infty}\right)^{5/2}\lesssim \epsilon^{1/2}\left(1+\int_1^t \frac{1}{s^{5/4}}ds\right)\lesssim \epsilon^{1/2}.
\end{align*}

\textit{Step 5. Estimate of the quadratic term in $v$.} 
Recalling that $\phi, \nabla \phi$ are bounded in $L^\infty(L^\infty)$, one has
\begin{align*}
& \left\|\int_0^\cdot U(\cdot-s)g_2(v(s),\phi(s))ds\right\|_{L^q(W^{1,r})}\lesssim \|g_2(v,\phi)\|_{L^{4/3}(W^{1,4/3})}\lesssim \||v|^2|\phi|^3\|_{L^{4/3}(W^{1,4/3})}\\\lesssim&
\left(\int_0^t \int |\phi|^4|v|^{8/3} + |\phi|^4|v|^{4/3}|\nabla v|^{4/3}+ |\phi|^{8/3}|v|^{8/3}|\nabla\phi|^{4/3} \right)^{3/4}\\\lesssim&
\left(\int_0^t \|\phi\|_{L^\infty}^{8/3}\left(\int |v|^2 + |\nabla v|^2 + |v|^4 + |\nabla v|^4\right)\right)^{3/4}\\\lesssim& \left(\left(\int_0^t \|\phi\|_{L^\infty}^{8/3}\int |v|^2 + |\nabla v|^2\right) + \left(\int_0^t \int |v|^4 + |\nabla v|^4\right)\right)^{3/4}\\\lesssim& \left(\|v\|^2_{L^\infty(H^1)}\left(1+\int_1^t \frac{1}{s^{8/6}}ds\right) + \|v\|_{L^4(W^{1,4})}^4\right)^{3/4} \lesssim h(t)^{3/2} + h(t)^3.
\end{align*}

\textit{Step 6. Conclusion.} Putting together Steps 2, 3, 4 and 5, there exists a  constant $D$, depending only on $M$, such that
\begin{equation}\label{estimativah}
h(t)\le D\left(\|v_0\|_{H^1} + \epsilon^{1/2}+ h(t)\epsilon^{1/4} + h(t)^{3/2} + h(t)^3 + h(t)^4 + h(t)^5\right).
\end{equation}
For $\epsilon$ sufficiently small, we arrive at
\begin{equation}
h(t)\lesssim  \|v_0\|_{H^1} +  \left(h(t)^{3/2} + h(t)^3 + h(t)^4 + h(t)^5\right)
\end{equation}
If $\|v_0\|_{H^1}$ is sufficiently small, then the above inequality implies $h(t)\in [0,h_0]\cup [h_1, \infty)$, for some $h_0<\delta, h_1$. Since $h(0)=0$, by continuity, one has $h(t)<\delta$, for all $t<T(u_0)$. The blow-up alternative then implies that $T(u_0)=\infty$. This implies that
$$
\|u-\phi\|_{L^\infty((0,\infty),H^1(\real^2))}\le \delta(\epsilon, M).
$$
Now notice that this property is also valid for $\tilde{u}$, since it is a solution of (NLS) with $v_0\equiv 0$. Hence
$$
\|u-\tilde{u}\|_{L^\infty((0,\infty),H^1(\real^2))}\le \|u-\phi\|_{L^\infty((0,\infty),H^1(\real^2))} + \|\tilde{u}-\phi\|_{L^\infty((0,\infty),H^1(\real^2))}\le 2\delta(\epsilon,M).
$$
\end{proof}

\section{Theory for the plane wave transform}

We recall the definition of the plane wave transform:
\begin{equation}
(Tf)(x,y) = \int f(x-cy,c)dc,\quad f\in C_0(\real^2).
\end{equation}

Now we state some simple properties of this transform, whose proof is quite straightforward.
\begin{lema}[Algebraic properties]
Fix any $f\in C_0(\real^2)$.
\begin{enumerate}
\item Translation property:
$$T(f(\cdot + z_0,\cdot + c_0))(x,y)= (Tf)(x+c_0y+z_0,y)$$
\item Scaling property:
$$T(f( \mu \cdot,\lambda \cdot))(x,y)=\frac{1}{\lambda}(Tf)\left(\mu x, \frac{\mu}{\lambda}y\right)$$
\item Monotonicity: if $g\in C_0(\real^2)$,
$$
f\le g \Rightarrow Tf\le Tg
$$
\item Derivation: if $f\in C^1_0(\real^2)$,
$$\nabla (Tf) =(T(f_z), -T(cf_z))$$
\end{enumerate}
\end{lema}

\begin{prop}[$L^p$ integrability, $p>2$]\label{integrabilidadeLp}
Fix $f\in C_0(\real^2)$. Then
$$
\|Tf\|_{L^p}^2 \le \int \frac{1}{|c-c'|^{2/p}}\|f(c)\|_{L_z^{p/2}}\|f(c')\|_{L_z^{p/2}}dcdc',\quad 2<p<\infty
$$
and
$$
\|Tf\|_{L^\infty_y(L^p_x)}\le \|f\|_{L^1_c(L^p_z)},\quad 1\le p\le\infty.
$$
Consequently, $T$ can be continuously extended in a unique way to $L^1_c(L^p_z)$, for any $1\le p< \infty$.
\end{prop}

\begin{proof}
Take $\psi\in C_0(\real^2)$. Then, for $q=p/2$,
\begin{align*}
\left|\int |(Tf)(x,y)|^2\psi(x,y) dxdy\right|&\le\int f(x-cy,c)f(x-c'y,c')\psi(x,y)dxdydcdc' \\&\le \int \left(\int |f(x-cy,c)f(x-c'y,c')|^qdxdy\right)^{1/q}\|\psi\|_{L^{q'}} dcdc' \\&\le 
\left(\int \frac{1}{|c-c'|^{1/q}}\left(\int |f(z,c)f(z',c')|^qdzdz'\right)^{1/q} dcdc'\right)\|\psi\|_{L^{q'}}\\&\le \left(\int \frac{1}{|c-c'|^{q}}\|f(c)\|_{L_z^{q}}\|f(c')\|_{L_z^{q}}dcdc'\right)\|\psi\|_{L^{q'}}.
\end{align*}
This implies that
$$
\|Tf\|_{L^p}^2=\||Tf|^2\|_{L^q}\le \int \frac{1}{|c-c'|^{2/p}}\|f(c)\|_{L_z^{q}}\|f(c')\|_{L_z^{q}}dcdc'.
$$
On the other hand, given $\phi\in C_0(\real)$, for a.e. $y\in\real$,
\begin{align*}
\left|\int Tf(x,y)\phi(x)dx\right|&\le \int |f(x-cy,c)||\phi(x)|dxdc \le \int \left(\int |f(x-cy,c)|^pdx\right)^{1/p}\|\phi\|_{L^{p'}} dc \\&= \|f\|_{L^1_c(L^p_z)}\|\phi\|_{L^{p'}}.
\end{align*}

\end{proof}
\begin{nota}\label{estaremL4}
As a consequence of the above result, for any $f\in C_0(\real^2)$,
$$\|Tf\|_{L^p(\real^2)}\lesssim \|f\|_{L^1_c(L^{p/2}_z)} + \|f\|_{L^\infty_c(L^{p/2}_z)}, \quad 2<p<\infty.
$$ 
Indeed,
\begin{align*}
&\int \frac{1}{|c-c'|^{2/p}}\|f(c)\|_{L_z^{p/2}}\|f(c')\|_{L_z^{p/2}}dcdc' = \int_{|c-c'|<1} \frac{1}{|c-c'|^{2/p}}\|f(c)\|_{L_z^{p/2}}\|f(c')\|_{L_z^{p/2}}dcdc' \\+\ & \int_{|c-c'|>1} \frac{1}{|c-c'|^{2/p}}\|f(c)\|_{L_z^{p/2}}\|f(c')\|_{L_z^{p/2}}dcdc'\le
\int_{|c-c'|<1} \frac{1}{|c-c'|^{2/p}}\|f(c)\|_{L_z^{p/2}}\|f\|_{L^\infty_c(L_z^{p/2})}dcdc' \\+\ & \int_{|c-c'|>1} \|f(c)\|_{L_z^{p/2}}\|f(c')\|_{L_z^{p/2}}dcdc' \le \left(\int_{|\tilde{c}|<1}\frac{1}{|\tilde{c}|^{2/p}}d\tilde{c}\right)\|f\|_{L^1_c(L^{p/2}_z)}\|f\|_{L^\infty_c(L_z^{p/2})} + \|f\|_{L^1_c(L^{p/2}_z)}^2.
\end{align*}
\end{nota}

\begin{exemplo}

It is a simple exercise to compute the transform of the characteristic function of the unit square: if $f=\mathds{1}_{[0,1]^2}$, then
$$
(Tf)(x,y)=\left\{\begin{array}{lll}
1, & 0\le x\le 1, & x-1\le y\le x\\
x/y, & 0\le x\le 1, & y\ge x\\
(x-1)/y, & 0\le x\le 1, & y\le x-1\\
(y-x)/y, & x\le 0, &  x-1\le y\le x\\
-1/y, & x\le 0, & y\le x-1\\
(y-x+1)/y, & x\ge 1, & x-1\le y\le x\\
1/y, & x\ge 1, &  y\ge x\\
0, & \mbox{ otherwise} & 
\end{array}\right.
$$
\begin{figure}[h]\label{transformada}
\centering
\includegraphics[width=7cm, height=5cm]{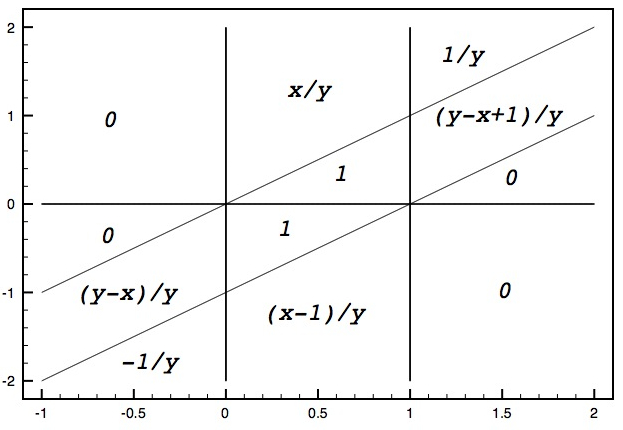}
\caption{The transform of $\mathds{1}_{[0,1]^2}$.}
\end{figure}

We claim that $Tf\notin L^2(\real^2)$. In fact, one has
$$
\int_{\{x\ge 1, y\ge x\}} |Tf|^2 dxdy = \int_1^\infty \left(\int_1^y \frac{1}{y^2}dx\right) dy = \int_1^\infty \frac{y-1}{y^2} dy = \infty.
$$
\end{exemplo}
\begin{cor}\label{naointegrabilidadeL2}
If $f\in L^1_c(L^\infty_z)$ satisfies $f\ge 0$ and $f\notequiv 0$, then $Tf\notin L^2(\real^2)$.
\end{cor}
\begin{proof}
Since $f\notequiv 0$ and $f\ge 0$, there exists $a>0$ and a square $Q\subset \real^2$ such that $f\ge a\mathds{1}_Q$, which implies that $Tf\ge aT\mathds{1}_Q$. Using the translation and scaling properties, one may write the transform of $\mathds{1}_Q$ in terms of $T\mathds{1}_{[0,1]^2}$, which does not belong to $L^2(\real^2)$.
\end{proof}

\begin{exemplo}
Consider $f=\mathds{1}_{[0,1]^2} - \mathds{1}_{[0,1]\times [1,2]}$. Then, using the expression for $T\mathds{1}_{[0,1]^2}$ and the translation property, it is easy to check that $Tf\in L^2(\real^2)$.
\end{exemplo}

\begin{exemplo}
Let's compute the transform of $f(z,c)=e^{-c^2-z^2}$:
\begin{align*}
\int e^{-c^2-(x-cy)^2} dc &= \int e^{-c^2-x^2+2cxy - c^2y^2}dc = e^{-x^2}\int e^{-(1+y^2)\left(c - \frac{xy}{1+y^2}\right)^2}e^{\frac{x^2y^2}{1+y^2}} dc\\
&=e^{-\frac{x^2}{1+y^2}}\frac{1}{\sqrt{1+y^2}}\int e^{-c^2} dc = 2\pi \frac{1}{\sqrt{4\pi(1+y^2)}}e^{-\frac{x^2}{1+y^2}}.
\end{align*}
It is interesting that $f$ is the kernel of the 2D-heat kernel in Fourier variables at time $t=1$, and that its transform is the 1D-heat kernel in physical variables at time $t=1+y^2$.
\end{exemplo}

\begin{prop}\label{propformula}
Given $f\in L^1(\real^2)$, one has
\begin{equation}\label{formula}
(Tf)(x,y)=\mathcal{F}_\xi^{-1}\left[(\mathcal{F}_{z,c}f)(\cdot,y\cdot)\right](x),\quad a.e.\ (x,y)\in\real^2
\end{equation}
where $\mathcal{F}_\xi$ denotes the Fourier transform in the $\xi$ variable. Consequently, one has the following inversion formula
$$
f(z,c)=\mathcal{F}_{\xi,\eta}^{-1}\left[(\mathcal{F}_x(Tf))\left(\xi,\frac{\eta}{\xi}\right)\right](z,c), \quad a.e.\ (z,c)\in\real^2
$$
\end{prop}
\begin{proof}
First of all, since $f\in L^1(\real^2)$,
one has $Tf\in L_y^\infty(L_x^1)$. For a.e. $y\in\real$, one may then take the Fourier transform of $Tf(\cdot,y)$:
$$
\mathcal{F}_x[Tf(\cdot,y)](\xi)=\int f(x-cy,c)e^{-2\pi i x\xi} dxdc = \int f(z,c)e^{-2\pi iz\xi-2\pi i cy\xi}dzdc = (\mathcal{F}_{z,c}f)(\xi,y\xi).
$$
\end{proof}

\begin{prop}
Take $f\in L^1_c(L^2_z)\cap \mathcal{F}(H^1(\real^2))$. For each $y\in \real$, define $\Gamma_y=\{(\xi,y\xi): \xi \in \real\}$ and let
$$
\pi_y: H^1(\real^2) \mapsto L^2(\Gamma_y)
$$
be the usual trace operator on $\Gamma_y$. Moreover, consider the isomorphism
$$
j: L^2(\Gamma_y)\mapsto L^2(\real)
$$
$$
j(f)(\xi) = f(\xi,y\xi) \quad \xi\in\real.
$$
and $\Pi_y=j\circ \pi_y$. Then, for a.e. $y\in\real$,
$$
\mathcal{F}_x[Tf(\cdot,y)] = \Pi_y\mathcal{F}_{z,c}f.
$$
\end{prop}
\begin{proof}
Take $f_n\to f$ in $L^1_c(L^2_z)\cap \mathcal{F}(H^1(\real^2))$, $f_n\in S(\real^2)$. In particular,
$$
Tf_n\to Tf \mbox{ in } L^\infty_y(L^2_x).
$$
Hence, for a.e $y\in \real^2$, using the previous result,
$$
\mathcal{F}_x[Tf(\cdot,y)] = \lim \mathcal{F}_x[Tf_n(\cdot,y)] = \lim \Pi_y\mathcal{F}_{z,c} f_n = \Pi_y\mathcal{F}_{z,c} f.
$$
\end{proof}

%\begin{cor}
%If  $f\in L^1_c(L^2_z)\cap \mathcal{F}(H^1(\real^2))$ is such that $Tf\equiv 0$, then $f\equiv 0$.
%\end{cor}
%\begin{proof}
%Since $Tf\equiv 0$, one has $\Pi_y\mathcal{F}_{z,c}f\equiv 0$ a.e. $y\in\real$, which clearly implies $\mathcal{F}_{z,c}f\equiv 0$.
%\end{proof}

\begin{cor}
[Transform of the product of two functions] For $f,g\in C_0(\real^2)$,
$$
T(fg)(x,y)=\int_\real \left(T\left(e^{-2\pi i kc}f\right)(x,y)\right)\left(T\left(e^{2\pi i kc}g\right)(x,y)\right)dk.
$$
\end{cor}
\begin{proof}
The result is a simple application of Proposition \ref{formula}. First of all, notice that
\begin{align*}
(\mathcal{F}_{z,c}f \ast \mathcal{F}_{z,c}g )(\xi, y\xi) &= \int (\mathcal{F}_{z,c}f)(l,k')(\mathcal{F}_{z,c}g)(\xi-l,y\xi-k')dldk' \\&= \int (\mathcal{F}_{z,c}f)(l,yl+k)(\mathcal{F}_{z,c}g)(\xi-l,y(\xi-l)-k)dldk \\&=\int (\mathcal{F}_{z,c}e^{-2\pi i ck}f)(l,yl)(\mathcal{F}_{z,c}e^{2\pi i ck}g)(\xi-l,y(\xi-l))dldk\\& = \int \left((\mathcal{F}_{z,c}e^{-2\pi i ck}f)(\cdot,y\cdot)\ast(\mathcal{F}_{z,c}e^{2\pi i ck}g)(\cdot,y\cdot)\right)(\xi)dk.
\end{align*}
Hence
\begin{align*}
T(fg)(\cdot,y) &= \mathcal{F}_\xi^{-1}\left((\mathcal{F}_{z,c}f \ast \mathcal{F}_{z,c}g )(\xi, y\xi)\right)\\&= \mathcal{F}_\xi^{-1}\left(\int \left((\mathcal{F}_{z,c}e^{-2\pi i ck}f)(\cdot,y\cdot)\ast(\mathcal{F}_{z,c}e^{2\pi i ck}g)(\cdot,y\cdot)\right)(\xi)dk\right)\\&= \int  \mathcal{F}_\xi^{-1}\left((\mathcal{F}_{z,c}e^{-2\pi i ck}f)(\cdot,y\cdot)\right)\mathcal{F}_\xi^{-1}\left((\mathcal{F}_{z,c}e^{2\pi i ck}g)(\cdot,y\cdot)\right)dk \\&= \int T\left(e^{-2\pi i ck}f\right)(\cdot, y)T\left(e^{2\pi i ck}g\right)(\cdot, y) dk.
\end{align*}
\end{proof}

\begin{prop}
[Parseval's identity for the plane wave transform]\label{parseval} For any $p\ge 1$, if $f\in L_c^1(L^p_z)$ and $g\in L^1_c(L^{p'}_z)$, one has
$$
\int (Tf)(x,y)g(x,y) dxdy = \int f(z,c)(Tg)(z,-c)dzdc.
$$
\end{prop}
\begin{proof}
If $f,g\in C_0(\real^2)$, the result follows from
\begin{align*}
\int (Tf)(x,y)g(x,y) dxdy &= \int f(x-cy,c)g(x,y)dxdydc = \int f(z,c)g(z+cy,y) dydcdz \\&= \int f(z,c)(Tg)(z,-c)dzdc.
\end{align*}
The general case is obtained by a density argument.
\end{proof}

\begin{cor}\label{unicarepresentacao}
If $f\in L^1_c(L^2_z)$ is such that $Tf\equiv0$, then $f\equiv 0$.
\end{cor}
\begin{proof}
Fix $h\in \mathcal{S}(\real^2)$, $h=h(z,c)$. Let $\psi\in C^\infty(\real)$ be such that $\psi\equiv 1$ in $[-1,1]$ and $\psi\equiv 0$ in $\real\setminus[-2,2]$. For any $\epsilon>0$, define
$$
\psi_\epsilon(\xi)=\psi\left(\frac{\xi}{\epsilon}\right),\quad h_\epsilon=h-\left(\mathcal{F}_{\xi}^{-1}\psi_\epsilon\right)\star_z h.
$$
It is clear that $\mathcal{F}_zh_\epsilon\in S(\real^2)$ has support outside the strip $\{|\xi|<\epsilon\}$. Furthermore,
$$
\|h-h_\epsilon\|_{L^\infty_c(L^2_z)}=\|\psi_\epsilon\mathcal{F}_z h\|_{L^\infty_c(L^2_z)}\to 0,\quad \epsilon\to 0.
$$
Setting
$$
g_\epsilon(\xi,\eta):=\left(\mathcal{F}_zh_\epsilon\right)\left(\xi,\frac{\eta}{\xi}\right),
$$
it follows that $g_\epsilon\in \mathcal{S}(\real^2)$ and so $\mathcal{F}_{\xi,\eta}^{-1}g_\epsilon\in \mathcal{S}(\real^2)$. By Proposition \ref{formula}, $T(\mathcal{F}_{\xi,\eta}^{-1}g_\epsilon)=h_\epsilon$. Using Proposition \ref{parseval},
\begin{align*}
\int f(z,c)h(z,c)dzdc &= \lim \int f(z,c)h_\epsilon(z,c)dzdc = \lim \int f(z,c)\left(T(\mathcal{F}_{\xi,\eta}^{-1}g_\epsilon)\right)(z,c)dzdc \\&= \lim \int Tf(z,-c)(\mathcal{F}_{\xi,\eta}^{-1}g_\epsilon)(z,c)dzdc = 0
\end{align*}
and so $f\equiv 0$.
\end{proof}

\begin{cor}[$L^2$ integrability]\label{integrabilidadeL2}
Take $f\in L^1(\real^2)$. Then
$$
\|Tf\|_{L^2(\real^2)}^2=\int \frac{1}{|\xi|}\left|\mathcal{F}_z f\right|^2(\xi, c)d\xi dc = \int \frac{1}{|\xi|}\left|\mathcal{F}_{z,c} f\right|^2(\xi, \eta)d\xi d\eta.
$$
Consequently, if the spectrum of $f$ lies in $(]-M,-\epsilon[\cup ]\epsilon,M[)\times ]-M,M[$, for some $M,\epsilon>0$, $Tf\in L^2(\real^2)$.
\end{cor}
\begin{proof}
This is a direct consequence of the formula for $T$ in terms of the Fourier transform.
\end{proof}
%
%Next, we extend formula \eqref{formula} for $f\in L^1_c(L^\infty_z)$. For any $1\le p,q\le \infty$, we define $L:\mathcal{F}_{z,c}\left(L^p_c(L^q_z)\right)\mapsto \mathcal{S}'(\real^2)$,
%$$
%\langle L\mathcal{F}_{z,c}f, g\rangle = \int f(z,c)(\mathcal{F}_xg)(z,y)e^{-2\pi icyx}dxdydzdc,\quad g\in \mathcal{S}(\real^2).
%$$
%
%Notice that, if $f\in L_c^1(L^1_z)$, then $\mathcal{F}_{z,c}f\in L^\infty_c(L^\infty_z)$ and
%$$
%\left(L\mathcal{F}_{z,c}f\right)(x,y)=\left(\mathcal{F}_{z,c}f\right)(x,yx)\quad a.e.\ x,y\in\real^2.
%$$
%\begin{prop}
%Given $f\in L^1_c(L^\infty_z)$, one has
%$$
%\mathcal{F}_x\left[Tf\right] = L\mathcal{F}_{z,c}f\quad \mbox{ in }\mathcal{S}'(\real^2).
%$$
%\end{prop}
%\begin{proof}
%Given $g\in \mathcal{S}(\real^2)$, one has
%$$
%\langle \mathcal{F}_x\left[Tf\right], g\rangle_{\mathcal{S}'\times\mathcal{S}} = \int (Tf)(x,y)\mathcal{F}_\xi \left[g(\cdot, y)\right] dxdy.
%$$
%The above integral makes all sense, since $Tf\in L^\infty_c(L^\infty_z)$ and $\mathcal{F}_\xi \left[g(\cdot, y)\right]\in \mathcal{S}(\real^2)$. Using Proposition \ref{parseval},
%\begin{align*}
%\langle \mathcal{F}_x\left[Tf\right], g\rangle_{\mathcal{S}'\times\mathcal{S}} &= \int f(z,c)T\left(\mathcal{F}_{\xi}g(\cdot,y)\right)(z,-c)dzdc \\&= \int f(z,c)g(\xi,y)e^{-2\pi i(z+cy)\xi}d\xi dydzdc \\&= \langle L\mathcal{F}_{z,c}f, g\rangle.
%\end{align*}
%\end{proof}

\begin{cor}
Take $f\in L^1(\real^2)$ and $g\in L^1(\real)$. Then
$$
T(f(\cdot, c) \ast g)=(Tf)(\cdot,y)\ast g.
$$
\end{cor}
\begin{proof}
By definition,
$$
T(f(\cdot, c) \ast g)(x,y)=\int (f(\cdot, c) \ast g)(x-cy)dc = \int f(x-cy-z,c)g(z)dzdc = ((Tf)(\cdot,y)\ast g)(x).
$$
\end{proof}

\begin{cor}[Convolution through the plane wave transform]\label{convolucao}
Take $f_1,f_2\in L^1(\real)$. If $f=f_1\otimes f_2$, that is, $f(z,c)=f_1(z)f_2(c)$, then
$$
(Tf)(x,1)=(f_1\ast f_2)(x).
$$
More generally, for any $y\neq 0$, one has
$$
(Tf)(x,y)=(f_1\ast \Theta_y f_2 )(x),\quad \Theta_y f_2(\cdot)=\frac{1}{|y|}f_2(\cdot/y).
$$
and, for $y=0$, $(Tf)(x,0)=\left(\int f_2 (c)\right)f_1(x)$.
\end{cor}
\begin{proof}
The result is trivial for $y=0$. Given $y\neq 0$, one has
\begin{align*}
(Tf)(x,y)&=\int f(x-cy,c) dc = \int f_1(x-cy)f_2(c)dc = \int \frac{1}{|y|}f_1(x-c')f_2\left(\frac{c'}{y}\right)dc' \\&= (f_1\ast \Theta_y f_2 )(x).
\end{align*}
\begin{nota}
Corollary \ref{convolucao} and Proposition \ref{integrabilidadeLp} give a new proof of Young's inequality:
$$
\|f_1\ast f_2\|_{L^p} \le \|Tf\|_{L^\infty_y(L^p_x)}\le \|f\|_{L^1_c(L^p_z)}=\|f_1\|_{L^p}\|f_2\|_{L^1}.
$$
Moreover, since, for a.e. $y\in\real$,
\begin{align*}
\left|\int (Tf(x,y)-Tf(x,0))\psi(x)dx\right| &= \left|\int (f(x,c)-f(x-cy,c))\psi(x)dcdx\right| \\&\le \int |f_2(c)|\|f_1(\cdot) - f_1(\cdot -cy)\|_{L^p}\|\psi\|_{L^{p'}}dc,
\end{align*}
when $y\to 0$, by Dominated Convergence Theorem, 
$$
\|Tf(x,y)-Tf(x,0)\|_{L^p_x}\le  \int |f_2(c)|\|f_1(\cdot) - f_1(\cdot -cy)\|_{L^p} dc \to 0
$$
which is a new way to prove convergence of mollifiers.
\end{nota}

\end{proof}

\section{Solving some linear PDE's using the plane wave transform}\label{resolverequacoes}

In this short section, we apply the plane wave transform to solve some classical equations.

\begin{exemplo}[The wave equation]
Consider the wave equation in $\real^2$
$$
u_{tt} - u_{xx} - u_{yy}=0.
$$
If one supposes that $u(x,y,t)=f(t,x-cy)$, one arrives to
$$
f_{tt} - (1+c^2)f_{zz} = 0,
$$
which may be explicitly solved:
$$
f(t,z,c)= \frac{f_0(z-\sqrt{1+c^2}t,c) + f_0(z+\sqrt{1+c^2}t,c)}{2} + \frac{1}{2\sqrt{1+c^2}}\int_{z-\sqrt{1+c^2}t}^{z+\sqrt{1+c^2}t} f_1(s,c) ds,
$$
Notice that we introduced on purpose the speed $c$ as an independent variable of $f$. Now, taking the transform of $f(t,\cdot)$, 
\begin{align*}
(Tf)(t,x,y)=\int &\frac{f_0(x-cy-\sqrt{1+c^2}t,c) + f_0(x-cy+\sqrt{1+c^2}t,c)}{2} \\&+\frac{1}{2\sqrt{1+c^2}}\left(\int_{x-cy-\sqrt{1+c^2}t}^{x-cy+\sqrt{1+c^2}t} f_1(s,c) ds\right)dc,
\end{align*}
one obtains a solution of the 2D wave equation. We remark that, unlike Poisson's formula for classical solutions, this representation does not involve derivatives of $f_1$.
\end{exemplo}

\begin{exemplo}[The Schrödinger equation]
Take the linear Schrödinger equation:
$$
iu_t + u_{xx} + u_{yy}=0.
$$
Introducing the plane wave ansatz $u(x,y,t)=f(t,x-cy)$, 
$$
if_t + (1+c^2)f_{zz} = 0.
$$
The solution is given by
$$
f(t,z,c)=\frac{1}{\sqrt{4i\pi (1+ c^2)t}}\int e^{\frac{i|z-w|^2}{4(1+ c^2)t}}f_0(w,c)dw
$$
which, through the plane wave transform, gives a family of solutions to the two dimensional Schrödinger equation.
$$
u(t,x,y)=\int\frac{1}{\sqrt{4i\pi (1+ c^2)t}} e^{\frac{i|x-cy-w|^2}{4(1+ c^2)t}}f_0(w,c)dwdc.
$$

\end{exemplo}

\begin{exemplo}[The heat equation]\label{calor}
We now consider a slightly different application of the plane wave transform. Take the heat equation in one dimension:
$$
u_t=u_{xx}.
$$
Here, instead of taking a plane wave in two spatial dimensions, we consider a plane wave in time and space. The ansatz $u(t,x)=f(x-ct)$ implies that
$$
-cf'=f'',\ i.e.,\ f(z,c)=A(c)e^{-cz} + B(c).
$$
Setting $B(c)=0$ and applying the transform $T$, one arrives to the following family of solutions to the heat equation:
$$
u(t,x)=\int A(c)e^{-cx+c^2t} dc.
$$
Actually, if one is given an initial condition $u(0,x)=u_0(x)$, then one observes that
$A$ is the inverse Laplace transform of $u_0$. Although this integral representation is certainly valid, the presence of the term $e^{-cx}$ implies a strong decay of $A$ at infinity. Now, \textit{if one replaces $c$ with $ic$ in the integral representation, one still obtains a solution to the heat equation}. This procedure makes no sense when one starts with the ansatz $u(t,x)=f(x-ct)$; however, the integral representation is still meaningful for complex values of $c$. Hence another family of solutions is
$$
u(t,x)=\int \tilde{A}(c)e^{-icx - c^2 t}dc.
$$
In this situation, one sees that $\tilde{A}$ is none other than the inverse Fourier transform of $u_0$ and so
$$
u(t,x)=\mathcal{F}_c\left(e^{-c^2t} \mathcal{F}_x^{-1}u_0\right),
$$
which is precisely the solution of the heat equation given by the Fourier transform. This integral representation was studied in \cite{chungyeom} as a generalization of the Fourier transform.
\end{exemplo}

\begin{exemplo}[The Schrödinger equation II]\label{schrodinger}
As in the previous example, one may derive a family of solutions to the 1D-Schrödinger equation using $T$ in both space and time:
$$
u(t,x)=\int A(c)e^{-icx + ic^2 t} dc.
$$
Givan an initial condition $u_0$, $A$ corresponds to the Fourier transform of $u_0$. The substitution $c\mapsto -ic$ gives the family
$$
u(t,x)= \int \tilde{A}(c)e^{-cx-ic^2t}dc
$$
which is connected to the Laplace transform.
\end{exemplo}

We finish with a result stating that the above classical arguments to reduce the dimension of the equations are valid in a functional setting. The result, although stated only for the linear Schrödinger equation, may also be extended to more equations.

\begin{prop}\label{semigrupoT}
Let $S_1$ and $S_2$ be the free Schrödinger groups in dimensions one and two, respectively. Given $f\in L^1_c(L^2_z)$, one has
$$
T(S_1((1+c^2)t)f(c)) = S_2(t)Tf \mbox{ in }\mathcal{S}'(\real^2).
$$
\end{prop}
\begin{proof}
Suppose that $f\in \mathcal{S}(\real^2)$. Then, from the formula of Proposition \ref{propformula},
\begin{align*}
&\mathcal{F}_{x}\left[S_2(t)Tf\right](\xi,y) = \mathcal{F}_{x}\left[ \mathcal{F}_{\xi,\eta}^{-1}\left(e^{-4\pi^2i(\xi^2+\eta^2)t}\mathcal{F}_{x,y}Tf\right)\right](\xi,y)\\
=\ &\mathcal{F}_{\eta}^{-1}\left[e^{-4\pi^2i(\xi^2 + \eta^2)t}\mathcal{F}_{x,y}Tf\right](\xi,y)
=\mathcal{F}_{\eta}^{-1}\left[e^{-4\pi^2i(\xi^2+\eta^2)t}\mathcal{F}_{y}\left((\mathcal{F}_{z,c}f)(\xi,y\xi)\right)\right](\xi,y)\\
=\ &\mathcal{F}_{\eta}^{-1}\left[e^{-4\pi^2i(\xi^2+\eta^2)t}\frac{1}{|\xi|}\left((\mathcal{F}_{z}f)(\xi,-\eta/\xi)\right)\right](\xi,y)
=\mathcal{F}_{\eta}\left[e^{-4\pi^2i(\xi^2+\eta^2)t}\frac{1}{|\xi|}\left((\mathcal{F}_{z}f)(\xi,\eta/\xi)\right)\right](\xi,y)\\
=\ &\mathcal{F}_{\eta}\left[e^{-4\pi^2i(1+\eta^2)\xi^2t}\left((\mathcal{F}_{z}f)(\xi,\eta)\right)\right](\xi,y\xi)=\mathcal{F}_{z,\eta}\left[\mathcal{F}_{\xi}^{-1}e^{-4\pi^2i(1+\eta^2)\xi^2t}\left((\mathcal{F}_{z}f)(\xi,\eta)\right)\right](\xi,y\xi)\\=\ & \mathcal{F}_{x}T(S_1((1+c^2)t)f(c))(\xi,y).
\end{align*}
The general case follows by a density argument.
\end{proof}
\section{Plane waves: the continuous case}
Consider
$$
X=\left\{ \phi\in L^\infty(\real^2): \phi=Tf,\ f\in L^1_c(H^1_z)\cap L^\infty_c(L^2_z)\right\}
$$
endowed with the induced norm
$$
\|\phi\|_X = \|f\|_{L^1_c(H^1_z)} + \|f\|_{L^\infty_c(L^2_z)}.
$$
Notice that such a norm is well-defined, by Corollary \ref{unicarepresentacao}. Furthermore, since $L^1_c(H^1_z)\hookrightarrow L^1_c(L^\infty_z)$, one has, by Proposition \ref{integrabilidadeLp} and Remark \ref{estaremL4}, 
$$
\|\phi\|_{L^4}, \|\phi\|_{L^\infty} \lesssim \|\phi\|_X.
$$
Now consider the subspace of $L^4(\real^2)$
$$
E=H^1(\real^2) + X,
$$
endowed with the semi-norm
$$
\|u\|_E=\inf_{u=v+\phi} \left\{\|v\|_{H^1} + \|\phi\|_X\right\}.
$$
\begin{lema}
The semi-norm $\|\cdot\|_E$ is a norm in $E$. Moreover, $(E,\|\cdot\|_E)$ is a Banach space and $E\hookrightarrow L^4(\real^2)$.
\end{lema}
\begin{proof}
Take $u\in E$. If $u=v+\phi$, $v\in H^1(\real^2)$, $\phi\in X$, one has
$$
\|u\|_{L^4}\le \|v\|_{L^4} + \|\phi\|_{L^4}\lesssim \|v\|_{H^1} + \|\phi\|_X.
$$
Taking the infimum on the right hand side,
$$
\|u\|_{L^4}\lesssim \|u\|_E.
$$
Hence, if $\|u\|_E=0$, $u=0$, which proves that $\|\cdot\|_E$ is a norm. We now prove that $E$, with the norm $\|\cdot\|_E$, is complete. Given $(u_n)_{n\in\nat}\subset E$, suppose that
$$
\|u_n-u_m\|_E\to 0,\quad n,m\to \infty.
$$
For each $k\in\nat$, there exist $n(k)\le m(k)$, increasing with $k$, such that
$$
\|u_n-u_m\|_E\le \frac{1}{2^k}, \quad n\ge n(k), m\ge m(k).
$$
Define
$$
\tilde{u}_k=u_{\max\{m(k), n(k+1)\}}
$$
and take $v_k\in H^1(\real^2)$ and $\phi_k\in X$, inductively, such that $\tilde{u}_k=v_k+\phi_k$ and
$$
\|v_k-v_{k+1}\|_{H^1} + \|\phi_k - \phi_{k+1}\|_{X}\le \|\tilde{u}_k-\tilde{u}_{k+1}\|_E + \frac{1}{2^k}\le \frac{1}{2^{k-1}}.
$$
This implies that $(v_k)_{k\in\nat}$ and $(\phi_k)_{k\in\nat}$ are Cauchy sequences in $H^1(\real^2)$ and $X$, respectively, and so
$$
v_k\to v\mbox{ in } H^1(\real^2),\quad \phi_k\to \phi\mbox{ in } X.
$$
Hence
$$
\|\tilde{u}_k-(v+\phi)\|_E\le \|v_k-v\|_{H^1} + \|\phi_k-\phi\|_X \to 0
$$
and so $\tilde{u}_k\to v+\phi$ in $E$. Since $(u_n)_{n\in\nat}$ is Cauchy in $E$, this implies that $u_n\to v+\phi$.
\end{proof}

\begin{lema}
Let $\{S_2(t)\}_{t\in\real}$ be the Schrödinger group in dimension two. Then 
$$
\|S_2(t)\|_{E\to E}=1,\ t\in \real.
$$
\end{lema}
\begin{proof}
Given $u\in E$, write $u=v+\phi$, $v\in H^1(\real^2)$, $\phi\in X$, $\phi=Tf$.  Then, by Proposition \ref{semigrupoT},
\begin{align*}
\|S_2(t)\phi\|_X &= \|S_1((1+c^2)t)f\|_{L^1_c(H^1_z)} +\|S_1((1+c^2)t)f\|_{L^\infty_c(L^2_z)} \\&= \|f\|_{L^1_c(H^1_z)} + \|f\|_{L^\infty_c(L^2_z)}= \|\phi\|_X.
\end{align*}
Hence
$$
\|S_2(t)v\|_{H^1} + \|S_2(t)\phi\|_X = \|v\|_{H^1} + \|\phi\|_X
$$
and, taking the infimum on both sides, one concludes the proof.
\end{proof}
For the sake of clarity, we define do we mean as a solution of (NLS) over $E$.
\begin{defi}\label{defisolucao}
Given $u_0\in E$ and $u\in C([0,T], E)$, $T>0$, we say that $u$ is a solution of (NLS) with initial data $u_0$ if $u$ satisfies the Duhamel formula
$$
u(t)=S_2(t)u_0 + i\lambda\int_0^t S_2(t-s)|u(s)|^\sigma u(s) ds, \ 0\le t\le T,
$$
where $S_2$ is the Schrödinger group in dimension two.
\end{defi}
\begin{lema}\label{solucaounica}
For any $\sigma\ge1$ and $u_0\in E$, if $u_1,u_2\in C([0,T], E)$, $T>0$, are two solutions of (NLS) with initial data $u_0$, then $u_1\equiv u_2$.
\end{lema}
\begin{proof}
Since $u_1, u_2$ are two solutions with the same initial data,
\begin{equation}\label{diferenca}
(u_1-u_2)(t)=i\lambda\int_0^t S_2(t-s)\left(|u_1(s)|^\sigma u_1(s)-|u_2(s)|^\sigma u_2(s)\right)ds, \ 0\le t\le T.
\end{equation}
Let $r\ge 2$ be such that $r'\ge 4/(\sigma+1)$. Take $p$ such that
\begin{equation}\label{restricoesrp}
r'p\ge 4,\quad r'p'\ge 4/\sigma.
\end{equation}
Then, since
$$
||u_1 |^\sigma u_1 - |u_2|^\sigma u_2|\le  C(|u_1|^\sigma + |u_2|^\sigma)|u_1-u_2|,\ 0\le s\le T,
$$
using Hölder's inequality,
$$
\|||u_1|^\sigma u_1 - |u_2|^\sigma u_2|\|_{L^{r'}}\le C (\|u_1\|_{L^{r'p'\sigma}}^\sigma+\|u_2\|_{L^{r'p'\sigma}}^\sigma)\|u_1-u_2\|_{L^{r'p}}.
$$
Define $\gamma$ and $q$ such that $(\gamma, r'p)$ and $(q,r)$ are admissible pairs, i.e.,
$$
\frac{2}{\gamma}=\frac{1}{2}-\frac{1}{r'p},\quad \frac{2}{q}=\frac{1}{2}-\frac{1}{r}.
$$
For any $J=[0,t]$, $0<t\le T$ and for $q=2r/(r-2)$, it follows that
$$
\|||u_1|^\sigma u_1 - |u_2|^\sigma u_2|\|_{L^{q'}(J,L^{r'})}\le C (\|u_1\|_{L^\infty(J,L^{r'p'\sigma})}^\sigma+\|u_2\|_{L^\infty(J,L^{r'p'\sigma})}^\sigma)\|u_1-u_2\|_{L^{q'}(J,L^{r'p})}
$$
Using Strichartz's estimates in \eqref{diferenca},
$$
\|u_1-u_2\|_{L^\gamma(J,L^{r'p})}\le C(\|u_1\|_{L^\infty(J,L^{r'p'\sigma})}^\sigma+\|u_2\|_{L^\infty(J,L^{r'p'\sigma})}^\sigma)\|u_1-u_2\|_{L^{q'}(J,L^{r'p})}.
$$
Since $E\hookrightarrow L^{r'p'\sigma}(\real^2)$,
$$
\|u_1-u_2\|_{L^\gamma(J,L^{r'p})}\le C(\|u_1\|_{L^\infty(J,E)}^\sigma+\|u_2\|_{L^\infty(J,E)}^\sigma)\|u_1-u_2\|_{L^{q'}(J,L^{r'p})}\le C'\|u_1-u_2\|_{L^{q'}(J,L^{r'p})}
$$
where $C'$ does not depend on $J$. The result now follows from \cite[Lemma 4.2.2]{cazenave}.
\end{proof}

\begin{lema}
For any $\sigma\ge 1$ and $u\in E$, $|u|^\sigma u\in H^{-1}(\real^2)$.
\end{lema}
\begin{proof}
Writing $u=v+\phi$,
$$
|u|^{\sigma+1}\lesssim |v|^{\sigma+1} + |\phi|^{\sigma+1}.
$$
Since $v\in H^1(\real^2)$, $v\in L^{\sigma+2}(\real^2)$ and so $|v|^{\sigma+1}\in L^{\frac{\sigma+2}{\sigma+1}}(\real^2)$. On the other hand, recalling that $\phi\in L^p$, for any $p\ge4$, $|\phi|^{\sigma+1}\in L^2(\real^2)$. Hence $|u|^{\sigma+1}\in L^2(\real^2) + L^{\frac{\sigma+2}{\sigma+1}}(\real^2)\hookrightarrow H^{-1}(\real^2)$.
\end{proof}
\begin{proof}[Proof of Theorem \ref{existenciacontinuo}]
Write $u_0=v_0+\phi_0$, $v_0\in H^1(\real^2)$, $\phi_0\in X$. Consider the initial value problem for $v\in C([0,T], H^1(\real^2))$,
\begin{equation}\label{pviv}
iv_t + v_{xx} + v_{yy} + \lambda|v+\phi|^\sigma(v+\phi)=0,\quad v(0)=v_0,\ \phi(t)=S_2(t)\phi_0.
\end{equation}
By the previous lemma, the nonlinear term is in $H^{-1}(\real^2)$. It follows from Kato's local existence theory (see, for example, \cite[Theorem 4.4.1]{cazenave}) that there exists $T=T(\|v_0\|_{H^1},\|\phi_0\|_X)>0$ and $v\in C([0,T], H^1(\real^2))$ solution of \eqref{pviv}, which depends continuously on $v_0$ and $\phi_0$. Setting $u(t)=v(t)+\phi(t)$, $0\le t< T$, it becomes clear that $u$ is a solution in $E$ of (NLS). Moreover, it is unique, by lemma \ref{solucaounica}.

Define
$$
T_{max}=\sup\{T>0: \mbox{ there exists } u\in C([0,T],E) \mbox{ solution of (NLS)},\ u(0)=u_0\}.
$$
If, for some $M>0$, $\|u(t)\|_E<M$ as $t\to T_{max}$, then, for each $t$, one may find $\tilde{v}(t)\in H^1(\real^2)$ and $\tilde{\phi}(t)\in X$ such that
$$
\|\tilde{v}(t)\|_{H^1} + \|\tilde{\phi}(t)\|_X\le \|u(t)\|_E + M< 2M, \quad  u(t)=\tilde{v}(t)+\tilde{\phi}(t).
$$
For $t_0$ sufficiently close to $T_{max}$,
$$
\overline{T}:=T(\|\tilde{v}(t_0)\|_{H^1}, \|\tilde{\phi}(t_0)\|_X)>T_{max}-t_0.
$$
Then one may solve \eqref{pviv} with initial data $v_0=\tilde{v}(t_0)$ and $\phi_0=\tilde{\phi}(t_0)$, thus obtaining a unique solution of (NLS) defined on $[T_{max}-t_0,\overline{T}]$. This implies that $u$ is extendible beyond $T_{max}$, a contradiction. Hence there exists a unique maximal solution $u\in C([0,T_{max}), E)$ of (NLS) and, if $T_{max}<\infty$,
$$
\|u(t)\|_E\to\infty,\quad t\to T_{max}.
$$

\end{proof}
\begin{proof}[Sketch of the proof of Theorem \ref{estabilidadecontinuo}]
The proof is very similar to that of the numerable case and again will only be done for $\sigma=4$. We start with some estimates for $S_2\phi_0$: using Propositions \ref{integrabilidadeLp} and \ref{semigrupoT},
\begin{align*}
\|S_2(t)Tf_0\|_{L^\infty}&=\|T(S_1((1+c^2)t)f_0)\|_{L^\infty}\le \|S_1((1+c^2)t)f_0\|_{L^1_c(L^\infty_z)}\\&\lesssim \frac{1}{t^{1/2}}\left\|\frac{f_0}{(1+c^2)^{1/2}}\right\|_{L^1_c(L^1_z)},\ \|f_0\|_{L^1_c(H^1_z)}
\end{align*}
\begin{align*}
\|\nabla S_2(t)Tf_0\|_{L^\infty}&=\|S_2(t)\nabla Tf_0\|_{L^\infty} = \|S_2(t)((f_0)_z,-c(f_0)_z)\|_{L^\infty}\\&\le \|S_1((1+c^2)t)((f_0)_z,-c(f_0)_z)\|_{L^1_c(L^\infty_z)} \\&\lesssim \frac{1}{t^{1/2}}\left\|(f_0)_z\right\|_{L^1_c(L^1_z)}, \|f_0\|_{L^1_c(H^2_z)} + \|cf_0\|_{L^1_c(H^2_z)}.
\end{align*}

The Theorem will be proved if we show that the unique solution $v$ of
$$
iv_t + v_{xx} + v_{yy} + \lambda|v+\phi|^\sigma(v+\phi)=0,\quad v(0)=v_0,\ \phi(t)=S_2(t)\phi_0
$$
is global and remains small in the $H^1$ norm, for all $t>0$. As before, we decompose the nonlinear term as
$$
\lambda|v+\phi|^\sigma(v+\phi)=\sum_{i=0}^5g_i(v,\phi),
$$
define $h$ as in \eqref{funcaoh} and apply some Strichartz estimates in the Duhamel's formula for $v$ in order to obtain inequality \eqref{estimativah}. The estimates for $i=1,...,5$ are \textit{mutatis mutandis} those that were derived for the numerable case, since one has control of the $W^{1,\infty}$ norm of $S_2\phi_0$. For the autonomous term, using the decay of $\|\phi\|_{W^{1,\infty}}$,
$$
\||\phi|^5\|_{L^1(H^1)}\le \int_0^t \|\phi(s)\|_{W^{1,\infty}}\|\phi(s)\|_{L^\infty}^2\|\phi(s)\|_4^2\lesssim \|\phi\|_{X}^2\left(1+\int_1^t \frac{1}{s^{3/2}}\right)\lesssim \epsilon^2
$$
\end{proof}

\begin{proof}[Proof of Theorem \ref{grandesdados}]

Let $\psi\in C^\infty(\real)$ be such that $0\le\psi\le 1$, $\psi\equiv 1$ in $[-1,1]$ and $\psi\equiv 0$ in $\real\setminus[-2,2]$. Take $g\in \mathcal{S}(\real)$ such that
$$
\|\mathcal{F}_\xi^{-1}\psi\|_{L^1}\|g\|_{W^{1,1}} + \|g\|_{H^2}<\sqrt{M},\  \|g\|_{H^1}<\sqrt{\epsilon(M)},\ |\mathcal{F}g(0)|>0.
$$

For any $\epsilon>0$, define
$$
\psi_\epsilon(\xi)=\psi\left(\frac{\xi}{\epsilon}\right),\quad g_\epsilon=g-\left(\mathcal{F}_{\xi}^{-1}\psi_\epsilon\right)\star_z g.
$$
Then
$$
\|g_\epsilon\|_{H^s} = \|(1+|\xi|^2)^{s/2}\mathcal{F}g_\epsilon\|_{L^2} \le \|(1+|\xi|^2)^{s/2}\mathcal{F}g\|_{L^2} = \|g\|_{H^s}, \ s=1,2.
$$
and
$$
\|g-g_\epsilon\|_{W^{1,1}}\le \|\mathcal{F}_\xi^{-1}\psi_\epsilon \|_{L^1}\|g\|_{W^{1,1}} \le \|\mathcal{F}_\xi^{-1}\psi\|_{L^1}\|g\|_{W^{1,1}} < \sqrt{M}.
$$
Finally, take $f\in \mathcal{S}(\real^2)$ such that
$$
\|f\|_{L^1\cap L^\infty}<\sqrt{\epsilon},\ \|cf\|_{L^1}<\sqrt{M},\ \|cf\|_{L^2}>\frac{K}{\||\xi|^{1/2}\mathcal{F}g\|_{L^2(|\xi|>1)}}.
$$
Defining $\phi_\epsilon = T(f\otimes g_\epsilon)$, it is now easy to check that $\phi_\epsilon$ satisfies the conditions of Theorem \ref{estabilidadecontinuo}. Moreover, by Corollary \ref{integrabilidadeL2}, for $\epsilon>0$ small
$$
\|\nabla \phi_\epsilon\|_{L^2} \ge \|cf\|_{L^2}\||\xi|^{1/2}\mathcal{F}g_\epsilon\|_{L^2}> \|cf\|_{L^2}\||\xi|^{1/2}\mathcal{F}g\|_{L^2(|\xi|>1)}>K.
$$
All that is left is to prove that, for $\epsilon>0$ small enough, $\|\phi_\epsilon\|_{L^2}>K$. This follows from
$$
\|\phi_\epsilon\|_{L^2}=\|f\|_{L^2}\left\|\frac{\mathcal{F}g_\epsilon}{|\xi|^{1/2}}\right\|_{2}\to \infty,\ \epsilon \to 0.
$$
\end{proof}

\begin{lema}
Suppose that $f\in L^1_c(H^1_z)\cap L^\infty_c(L^2_z)$. If
\begin{equation}\label{cond1}
zf, (1+c^2)f_z\in L^1_c(L^2_z)\cap L^\infty_c(L^2_z).
\end{equation}
then, for some $C=C(f)$,
$$
\|S_2(t)Tf\|_{L^p}\lesssim \left\{\begin{array}{ll}
C & 2\le p\le \infty\\
C(1+t^2) & 1<p<2
\end{array}\right..
$$
\end{lema}
\begin{proof}
From Proposition \ref{semigrupoT} and Remark \ref{estaremL4},
\begin{align*}
&\|S_2(t)Tf\|_{L^4} = \|T(S_1((1+c^2)t)f(c))\|_{L^4} \\\lesssim\ & \|S_1((1+c^2)t)f(c)\|_{L^1_c(L^2_z)} + \|S_1((1+c^2)t)f(c)\|_{L^\infty_c(L^2_z)}  = \|f\|_{L^1_c(L^2_z)} + \|f\|_{L^\infty_c(L^2_z)}.
\end{align*}
Furthermore, from Proposition \ref{integrabilidadeLp}, 
$$
\|S_2(t)Tf\|_{L^\infty} = \|T(S_1((1+c^2)t)f(c))\|_{L^\infty} \lesssim \|f\|_{L^1_c(H^1_z)}. 
$$
The estimate for $2<p<\infty$ follows by interpolation. Now take $1<p<2$. By Proposition \ref{integrabilidadeLp},
$$
\|S_2(t)Tf\|_{L^{2p}}\le \int \frac{1}{|c-c'|}\|S_1((1+c^2)t)f(c)\|_{L^p_z}\|S_1((1+(c')^2t))f(c')\|_{L^p_z} dcdc'.
$$
Using Hölder's inequality,
\begin{align*}
\|S_1((1+c^2)t)f(c)\|_{L^p_z}&\le \left\|\frac{1}{1+|z|}\right\|_{L^{\frac{2p}{2-p}}}\|(1+|z|)S_1((1+c^2)t)f(c)\|_{L^2_z}\\&\lesssim \|f(c)\|_{L^2_z} + \|zS_1((1+c^2)t)f(c)\|_{L^2_z}.
\end{align*}
Recall the classical estimate, valid a.e. for $c\in\real$:
$$
\|zS_1((1+c^2)t)f(c)\|_{L^2_z} \le \|zf(c)\|_{L^2_z} + 2(1+c^2)t\|f_z(c)\|_{L^2_z}.
$$
Hence, setting
$$
\Theta(c,t)= \|f(c)\|_{L^2_z} + \|zf(c)\|_{L^2_z} + 2(1+c^2)t\|f_z(c)\|_{L^2_z},
$$
we have
$$
\|S_2(t)Tf\|_{L^{2p}}\lesssim \int \frac{1}{|c-c'|}\Theta(c,t)\Theta(c',t)dcdc'.
$$
By a analogous reasoning to that of Remark \ref{estaremL4} and \eqref{cond1},
$$
\|S_2(t)Tf\|_{L^{2p}}\lesssim 1+t^2.
$$
%Finally, the estimate for $\|\nabla S_2(t)Tf\|_{L^{2p}}$ can be obtained in a very similar way, simply by recalling that
%$$
%\nabla S_2(t)Tf = (S_2(t)T(f_z), S_2(t)T(-cf_z)).
%$$
\end{proof}
\begin{proof}[Proof of Theorem \ref{gwp}]
The theorem will be proved once we show that the unique solution of
\begin{equation}\label{equav}
iv_t + v_{xx} + v_{yy} + \lambda|v+\phi|(v+\phi)=0,\quad v(0)=v_0,\ \phi(t)=S_2(t)\phi_0
\end{equation}
is globally defined. In the following, we perform formal calculations which can be justified with suitable regularization arguments. Multiplying \eqref{equav} by $\bar{v}$, integrating over $\real^2$ and taking the imaginary part, we obtain
$$
\frac{1}{2}\frac{d}{dt}\|v(t)\|_{L^2}^2 + \lambda \parteim \int |v+\phi|(v+\phi)\bar{v} = 0.
$$
and so
$$
\frac{1}{2}\frac{d}{dt}\|v(t)\|_{L^2}^2\lesssim \|\phi\|_{L^\infty} \|v\|_{L^2}^2 + \|v\|_{L^2}\|\phi\|_{L^4}^2
$$
Since, by the previous lemma, $\|\phi\|_{L^4}$ and $\|\phi\|_{L^\infty}$ are uniformly bounded, Gronwall's lemma implies that
$$
\|v(t)\|_{L^2}\lesssim C(t),
$$
where $C:\real^+\to\real$ is a non decreasing continuous function. Now, multiplying \eqref{equav} by $\overline{v}_t$, integrating and taking the real part,
$$
\frac{d}{dt}\left(\frac{1}{2}\|\nabla v(t)\|_{L^2}^2 - \frac{\lambda}{3}\|v+\phi\|_{L^3}^3\right) =  \parteim \lambda\int \nabla\left(|v+\phi|(v+\phi)\right)\cdot \nabla \bar{\phi}. 
$$
Recalling that $\nabla \phi(t)= S_2T(f_z,-cf_z)$ and applying the previous lemma,
\begin{align*}
&\frac{d}{dt}\left(\frac{1}{2}\|\nabla v(t)\|_{L^2}^2 - \frac{\lambda}{3}\|(v+\phi)(t)\|_{L^3}^3\right)\\\lesssim\ & \int |\nabla \phi(t)|^2|\phi(t)| + |\nabla\phi(t)||\phi(t)||v(t)| + |\nabla \phi(t)|^2|v(t)| + |\nabla\phi(t)||\nabla v(t)||v(t)| \\ \lesssim\ & \|\phi(t)\|_{L^4}\|\nabla\phi(t)\|_{L^{8/3}}^2 + \|\phi(t)\|_{L^4}\|\nabla \phi(t)\|_{L^4}\|v(t)\|_{L^2}\\&+\|\nabla \phi(t)\|_{L^4}^2\|v(t)\|_{L^2} + \|\nabla\phi(t)\|_{L^\infty}\|\nabla v(t)\|_{L^2}\|v(t)\|_{L^2}\\\lesssim\ & (1+t)^4 + C(t) + C(t)\|\nabla v(t)\|_2.
\end{align*}
Define the energy
$$
E(t)=\frac{1}{2}\|\nabla v(t)\|_{L^2}^2 - \frac{\lambda}{3}\|v+\phi\|_{L^3}^3.
$$
From Gagliardo-Nirenberg's inequality,
$$
\|v+\phi\|_{L^3}^3\lesssim \|v\|_{L^3}^3 + \|\phi\|_{L^3}^3\lesssim \|\nabla v\|_{L^2}\|v\|_{L^2}^2 + (1+t)^6\lesssim C(t)^2\|\nabla v\|_{L^2} + (1+t)^6.
$$
Let $T_{max}$ be the maximal time of existence of the solution of \eqref{equav}. Then, for any $0<t_0<t<T_{max}$, 
$$
\|\nabla v(t)\|_{L^2}^2\lesssim |E(t_0)| + C(t)^2\|\nabla v(t)\|_{L^2} + (1+t)^6 + \int_{t_0}^t \left((1+s)^4 + C(s) + C(s)\|\nabla v(s)\|_2\right)ds.
$$
Suppose, by contradiction, that $T_{max}<\infty$. Since $\|\nabla v(t)\|_{L^2}\to\infty$ as $t\to T_{max}$, one has, for $t>t_1$ sufficiently close to $T_{max}$,
$$
\|\nabla v(t)\|_{L^2}^2\lesssim \int_{t_0}^t C(s)\|\nabla v(s)\|_2 ds.
$$
However, this implies that $\|\nabla v(t)\|_{L^2}^2$ is bounded on $[t_1,T_{max}]$, which is absurd.
\end{proof}

\section{Further comments}

The plane wave transform has many interesting properties. Section 4 should be regarded as a starting point for the application of this transform to obtain new local well-posedness results for equations such as the wave equation, the heat equation and KdV equation. Furthermore, Theorem \ref{grandesdados} points out that this theory is not simply an appendix to the usual $H^s$ local well-posedness theory: it gives new information about $H^s$ solutions and may shine some new light on global existence criteria for these equations.

It is important to recall that the plane transform is built upon solutions of the transport equation. Other equations may give birth to other transforms: for example, consider
$$
\omega\cdot \nabla u = iu,\quad u=u(x,y),\ \omega=(\omega_1,\omega_2)\in \real^2,\ |\omega|=1.
$$
Solutions of this equation are of the form $u(x,y)=e^{i\omega\cdot(x,y)}f(\omega^\bot\cdot (x,y))$, $\omega^\bot=(-\omega_2,\omega_1)$, that is, functions that are periodic in the $\omega$ direction. By varying the direction of $\omega$, we may arrive to a similar construction to that of the transform $T$. It would be especially important to find a construction which can be applied to other spatial domains.

One may also try to extend the plane wave transform to more spatial variables, either by applying only to the two first variables or by iteration (applying to the first and to the $i$-th variable). The functional tools of section 3 should be easily applicable to either extension.

\section{Acknowledgements}
The authors were partially suported by Fundação para a Ciência e Tecnologia, through the grant UID/MAT/04561/2013. The first author was also partially supported by Fundação para a Ciência e Tecnologia, through the grant SFRH/BD/96399/2013. The authors thank Rémi Carles for important discussions and comments.

\appendix

\section{Appendix}
In this appendix, we give another proof of
$$
Tf\equiv 0 \Rightarrow f\equiv 0.
$$
This proof is based on arguments that do not involve the expression of $T$ in terms of the Fourier transform. For the sake of simplicity, we shall assume that $f$ is good enough so that all integrations are automatically justified.
\begin{prop}
Given $f\in C(\real^2)$ with exponential decay at infinity, if $Tf\equiv 0$, then $f\equiv 0$.
\end{prop}
\begin{proof}
\textit{Step 1. Reduction to radial case.} Since $Tf\equiv 0$,
$$
\int f(x-cy,c)dc = 0, \ (x,y)\in\real^2.
$$
Then, given any $h\in \real$,
$$
\int_0^1 \int f( x + hz - cy,c) dc dz =0,
$$
which means that the integration of $f$ over any non-vertical strip is null.
Define
$$
\tilde{f}(z,c)=\frac{1}{2\pi}\int_0^{2\pi} f(A_\theta(z,c)) d\theta,\quad A_\theta(z,c)=(c\cos\theta - z\sin\theta, c\sin\theta + z\cos\theta)
$$
Then, for any $\epsilon,\delta>0$ and any strip $S= ]\epsilon-\delta,\epsilon+\delta[\times\real$,
$$
\int_S \tilde{f}(z,c) dcdz = \frac{1}{2\pi}\int_0^{2\pi} \int_S f(A_\theta(z,c))dc d\theta = \frac{1}{2\pi}\int_0^{2\pi} \int_{A_\theta^{-1}S}f(z,c) dcdz = 0,
$$
since the rotation of a strip is still a strip. In particular, the function
$$
\epsilon\mapsto \int \tilde{f}(\epsilon,c)dc
$$
is zero. Since $\tilde{f}$ is a radial function, $\tilde{f}(\epsilon,c)$ is an even function, and so
$$
\int_0^\infty \tilde{f}(\epsilon,c) dc =0,\ \epsilon>0.
$$
Hence
$$
0=\int_0^{2\pi} \int_0^\infty \tilde{f}(\epsilon,c) dc d\theta = \int_{\real^2\setminus B_\epsilon(0)} \frac{\tilde{f}(z,c)}{\sqrt{c^2+z^2-\epsilon^2}}dzdc
$$
which implies, in polar coordinates, that
\begin{equation}\label{zeroepsilon}
\int_{\epsilon}^{\infty} \frac{\tilde{f}(r)r}{\sqrt{r^2-\epsilon^2}}dr =0, \ \epsilon>0.
\end{equation}
\textit{Step 2. $\tilde{f}$ must be zero.} We write $\tilde{f}=f^+-f^-$, with $f^+,f^-\ge 0$ and $f^+f^-\equiv 0$. Fix $a\in \real^+$. W.l.o.g., let $(a_n)_{n\in\nat}\subset \real^+$ be a strictly increasing sequence with $a_0=a$ and such that
\begin{equation}\label{sinalf}
\tilde{f}\big|_{[a_{2k},a_{2k+1}]}\ge 0,\quad \tilde{f}\big|_{[a_{2k+1},a_{2k+2}]}\le 0, k\in\nat.
\end{equation}
There are two possibilities: either $a_n$ is bounded or not. We consider the case where $a_n$ is not bounded, the other being quite similar. Define
$$
h_n(r)=\frac{r}{\sqrt{r^2-a_n^2}}\mathds{1}_{(a_n,+\infty)}
$$
and
$$
S_N(r)=\sum_{n=0}^N (-1)^n h_n(r), N\in \nat.
$$
Since $h_n(r)\le h_{n+1}(r)$ for $r>a_{n+1}$, it is easy to check that, for $N$ even, 
\begin{equation}\label{sinalS}
S_N\big|_{(a_{2k},a_{2k+1})}> 0,\quad S_N\big|_{(a_{2k+1},a_{2k+2})}< 0,\quad S_N\big|_{(a_{N},\infty)}> 0, \quad  k\in\nat, 2k+2\le N
\end{equation}
This implies that, for all $N$ even, $f^+S_N\ge 0$. On the other hand, the same reasoning shows that, for all $N$ odd, $f^-S_N\le 0$. Given $k\in \nat$, it follows from \eqref{zeroepsilon} and from the decay of $f$ that
\begin{align*}
\left|\int_0^\infty f^+S_{2k} - \int_0^\infty f^-S_{2k+1}\right| &= \left|\int_0^\infty \sum_{n=1}^{2k} (-1)^n(f^+ - f^-)h_n + \int_0^\infty f^-h_{2k+1}\right| \\&= \left| \sum_{n=1}^{2k} (-1)^n\int_0^\infty fh_n + \int_0^\infty f^-h_{2k+1}\right| = \left| \int_0^\infty f^-h_{2k+1}\right|\to 0,\ k\to\infty
\end{align*}
Since, for any $k$,
$$
\int_0^\infty f^+S_{2k} \ge 0,  \int_0^\infty f^-S_{2k+1}\le 0,
$$
it follows that
$$
\int_0^\infty f^+S_{2k},\ \int_0^\infty f^-S_{2k+1}\to 0,\ k\to \infty.
$$
Set $S(r)=\lim S_N(r)$. By Monotone Convergence Theorem, the functions
$f^+S$ and $-f^-S$ are integrable, positive and have integral over $[0,\infty)$ equal to 0. Hence $f^+S,f^-S\equiv 0$. It follows from \eqref{sinalf} and \eqref{sinalS} that $f^+(r),f^-(r)=0$ for $r>a$. Since $a>0$ is arbitrary, we conclude that $\tilde{f}\equiv 0$.

\textit{Step 3. Conclusion.} Since $\tilde{f}$ is the average of $f$ over any circle centered at the origin and $\tilde{f}\equiv 0$, it follows that $f(0,0)=0$. To prove that $f(z_0,c_0)=0$ for any $(z_0,c_0)\in\real^2$, consider
$$
f_{(z_0,c_0)}(z,c)=f(z+z_0,c+c_0).
$$
Then $Tf_{(z_0,c_0)}\equiv 0$ and therefore, by the previous steps, $f(z_0,c_0)=f_{(z_0,c_0)}(0,0)=0$.
\end{proof}

\begin{nota}
It follows directly from this proof that the integral of $f$ over the collection of all strips in $\real^2$ determines uniquely the values of $f$. This implies that the integral of $f$ over all strips determines the integral of $f$ over all balls. This is not trivial at all, since the value of the latter cannot be obtained from the first ones using algebraic set relations. Furthermore, it is easy to check that, if one starts with a collection of strips with a finite number of possible directions, such a result is no longer valid.
\end{nota}

\small
\noindent \textsc{Sim\~ao Correia}\\
CMAF-CIO and FCUL \\
\noindent Campo Grande, Edif\'icio C6, Piso 2, 1749-016 Lisboa (Portugal)\\
\verb"sfcorreia@fc.ul.pt"\\

\small
\noindent \textsc{Mário Figueira}\\
CMAF-CIO and FCUL \\
\noindent Campo Grande, Edif\'icio C6, Piso 2, 1749-016 Lisboa (Portugal)\\
\verb"msfigueira@fc.ul.pt"\\

\begin{thebibliography}{99}

\bibitem{cazenave} T. Cazenave, \emph{Semilinear Schrödinger Equations}, Courant Institute of Mathematical Sciences (2003)

\bibitem{cazenaveweissler} T. Cazenave, F. Weissler, \emph{Rapidly decaying solutions of the nonlinear Schrödinger equation}, Commun. Math. Phys. 147, No.1, 75-100 (1992)

\bibitem{chungyeom} S.-Y. Chung, Y. Yeom, \emph{An Integral Transformation and its Applications to Harmonic Analysis on the Space of Solutions of Heat Equation}, Publ RIMS, Kyoto Univ. 35, 737-755  (1999)

\bibitem{correiafigueira} S. Correia, M. Figueira, \emph{The hyperbolic nonlinear Schrödinger equation}, arXiv:1510.08745

\bibitem{gallo} C. Gallo, \emph{Schrödinger group on Zhidkov spaces}, Adv. Diff. Eq.
 9, No. 5-6, 509-538 (2004)

\bibitem{glassey} R. T. Glassey, \emph{On the blowing up of solutions to the Cauchy problem for
nonlinear Schrödinger equations}, J. Math. PhEs. 18, 7794-7797 (1977)

\bibitem{hayashi} N. Hayashi, P. I. Naumkin, \emph{On the quadratic nonlinear Schrödinger equation in three space
dimensions}. Internat. Math. Res. Notices 3, 115-132 (2000)

\bibitem{hayashisaitoh} N. Hayashi, S. Saitoh, \emph{Analyticity and global existence of small solutions to some nonlinear Schrödinger equations}, Comm. Math. Phys.
129, No.1, 27-41 (1990)

\bibitem{kato} T. Kato, \emph{On nonlinear Schrödinger equations}, Ann. Inst. Henri Poincaré, Phys. Théor. 46, 113-129 (1987)

\bibitem{kato1} T. Kato, \emph{On nonlinear Schrodinger equations. II. $H^s$-solutions and unconditional
well-posedness}, J. Anal. Math. 67, 281-306 (1995)

\bibitem{nakamuraozawa} M. Nakamura, T. Ozawa, \emph{Small solutions to nonlinear Schrcidinger equations in the Sobolev
spaces} J. Anal. Math. 81, 305-329 (2000)

\bibitem{sulem} C. Sulem, J.P. Sulem, \emph{Nonlinear Schrödinger Equations: Self-Focusing and Wave Collapse}, Applied Mathematical Sciences 139, Springer (1999)

\bibitem{tao} T. Tao, \emph{Nonlinear Dispersive Equations: Local and Global Analysis}, CBMS Regional Conference Series in Mathematics, vol. 106 (2006)

\end{thebibliography}
\end{document}